\documentclass[11pt]{amsart}

\makeindex

\usepackage{amsfonts,graphics,amsmath,amsthm,amsfonts,amscd,amssymb,amsmath,latexsym,multicol,euscript}
\usepackage{epsfig,url}
\usepackage{flafter}
\usepackage{hyperref}
\usepackage[all,cmtip,line]{xy}
\usepackage{array}
\usepackage[english]{babel}
\usepackage{overpic}
\usepackage{subfig}
\usepackage{multirow}

\usepackage{pgf,tikz}
\usepackage{mathrsfs}
\usetikzlibrary{arrows}

\usepackage{wrapfig}

\usepackage{enumitem}

\theoremstyle{definition}

\theoremstyle{theorem}
\newtheorem*{questionwn}{Question}
\newtheorem{theoremalpha}{Theorem}
\theoremstyle{corollary}

\definecolor{uuuuuu}{rgb}{0.26666666666666666,0.26666666666666666,0.26666666666666666}

\newtheorem{theorem}{Theorem}[section]
\newtheorem{main theorem}{Main Theorem}
\newtheorem{proposition}[theorem]{Proposition}

\newtheorem{lemma}[theorem]{Lemma}

\newtheorem{theorem*}{Theorem}
\newtheorem{corollary*}[theorem*]{Corollary}
\newtheorem{conjecture*}[theorem*]{Conjecture}

\theoremstyle{definition}
\newtheorem{example}[theorem]{Example}
\newtheorem{definition}[theorem]{Definition}

\newtheorem{definition-lemma}[theorem]{Definition-Lemma}

\newtheorem{remark}[theorem]{Remark}

\newcommand\R{\mathbb{R}}
\newcommand\Q{\mathbb{Q}}

\newcommand\C{\mathbb{C}}
\renewcommand\P{\mathbb{P}}

\newcommand\eps{\varepsilon}
\renewcommand\epsilon{\varepsilon}

\newcommand{\mc}{\mathcal}

\DeclareMathOperator{\codim}{codim}

\DeclareMathOperator{\im}{Im}
\DeclareMathOperator{\ord}{ord}
\DeclareMathOperator{\mult}{mult}

\DeclareMathOperator{\Supp}{Supp}

\DeclareMathOperator{\vol}{vol}

\DeclareMathOperator{\SB}{SB}

\newcommand{\bm}{\mathbf B_-}  

\newcommand{\bp}{\mathbf B_+}  
\newcommand{\okbd}{\Delta}

\newcommand{\oklim}{\Delta^{\lim}}

\def\axis{\operatorname{axis}}

\begin{document}

\title[Local numerical equivalences and Okounkov bodies in higher dimensions]{Local numerical equivalences and Okounkov bodies \\in higher dimensions}

\author{Sung Rak Choi} \address{Department of Mathematics, Yonsei University, Seoul, Korea} \email{sungrakc@yonsei.ac.kr}

\author{Jinhyung Park}
\address{Department of Mathematics, Sogang University, Seoul, Korea}
\email{parkjh13@sogang.ac.kr}

\author{Joonyeong Won} \address{Center for Mathematical Challenges, Korea Institute for Advanced Study, Seoul, Korea} \email{leonwon@kias.re.kr}

\date{\today}
\thanks{S. Choi and J. Park were partially supported by the National Research Foundation of Korea (NRF-2016R1C1B2011446). J. Won was supported by the National Research Foundation of Korea (NRF-2020R1A2C1A01008018) and a KIAS Individual Grant (SP037003) via the Center for Mathematical Challenges at Korea Institute for Advanced Study.}

\begin{abstract}
We continue to explore the numerical nature of the Okounkov bodies focusing on the local behaviors near given points.
More precisely, we show that the set of Okounkov bodies of a pseudoeffective divisor with respect to admissible flags centered at a fixed point determines the local numerical equivalence class of divisors which is defined in terms of refined divisorial Zariski decompositions. Our results extend Ro\'{e}'s work \cite{Roe} on surfaces to higher dimensional varieties although our proof is essentially different in nature.
\end{abstract}

\maketitle


\section{Introduction}
Lazarsfeld--Musta\c{t}\u{a} \cite{lm-nobody} and Kaveh--Khovanskii \cite{KK} independently introduced Okounkov bodies based on the pioneering works of Okounkov \cite{O1, O2}. Let $X$ be a smooth projective variety of dimension $n$, and $D$ be a divisor on $X$.  An \emph{admissible flag} $Y_\bullet$ on $X$ is a sequence of irreducible subvarieties
$$
Y_\bullet: X=Y_0\supseteq Y_1\supseteq \cdots\supseteq Y_{n-1}\supseteq Y_n=\{ x\}
$$
where each $Y_i$ is of codimension $i$ in $X$ and is smooth at the point $x$. Using a given admissible flag $Y_\bullet$ on $X$, we can define a valuation-like function $\nu_{Y_\bullet} \colon |D|_{\R}\to \R^{n}$. The \emph{Okounkov body} $\okbd_{Y_\bullet}(D)$ of a divisor $D$ with respect to an admissible flag $Y_\bullet$ is defined as the closure of the convex hull of the image $\nu_{Y_\bullet}(|D|_{\R})$ in the Euclidean space $\R^n$. See Section \ref{okbd_sec} for the precise construction of the Okounkov body.

By \cite[Proposition 4.1]{lm-nobody} and \cite[Theorem A]{Jow}, if $D, D'$ are big divisors on a smooth projective variety $X$, then $D$ is numerically equivalent to $D'$ if and only if $\okbd_{Y_\bullet}(D)=\okbd_{Y_\bullet}(D')$ for all admissible flags $Y_\bullet$ on $X$. This result is extended to the limiting Okounkov bodies of pseudoeffective divisors in \cite[Theorem C]{CHPW-okbd I}. Therefore, in principle, all the numerical information of a pseudoeffective divisor $D$ must be contained in the set of Okounkov bodies of $D$ with respect to all the admissible flags. Based on these results, there have been extensive and thorough studies of asymptotic numerical positivity of divisors via Okounkov bodies.
The recent results tell us that the ``local'' numerical properties such as moving Seshadri constant $\epsilon(||D||;x)$ can be computed from the Okounkov bodies $\okbd_{Y_\bullet}(D)$ by fixing $Y_n$ at a given point $x$ of $X$ (see \cite{CHPW-asyba}, \cite{AV-loc pos, AV-loc pos3}).
One can also extract the ``global'' numerical properties such as ampleness, nefness, and the asymptotic base loci $\bp(D)$, $\bm(D)$ from the Okounkov bodies $\okbd_{Y_\bullet}(D)$ by varying admissible flags $Y_\bullet$ (see \cite{CHPW-asyba}, \cite{AV-loc pos, AV-loc pos2, AV-loc pos3}). We remark that even in this global case, the results are based on the analysis of Okounkov bodies for admissible flags $Y_\bullet$ with a fixed center $Y_n=\{ x \}$.

Now, it is natural to wonder what other local information can be obtained from the set of Okounkov bodies for $Y_\bullet$ with a fixed center $Y_n$. In other words, we may ask the following:
\begin{questionwn}
What local numerical properties of a pseudoeffective divisor $D$ are precisely contained in the set of Okounkov bodies of $D$ with respect to admissible flags centered at a given fixed point?
\end{questionwn}

This paper is devoted to answering this question in full generality.
Naturally, the Okounkov bodies described in question are supposed to give some kind of  \emph{local} numerical data of the divisor $D$ at the given point. The local numerical properties in question will be clarified by defining the local numerical equivalence relation on pseudoeffective divisors using the divisorial Zariski decomposition. We then prove that the local numerical equivalence class of a divisor at a given point can be read off from the Okounkov bodies constructed with the so-called \emph{induced} admissible flags centered at the point. In the surface case, Ro\'{e} in \cite{Roe} used the notion of clusters at infinitely near points to extract the local numerical equivalence of divisors from the Okounkov bodies. Rather than generalizing Ro\'{e}'s notion of clusters into higher dimensions for our purpose, we simply analyze the Okounkov bodies varying the admissible flags. We believe that this approach is more natural. 

\medskip

Turning to the details, let $X$ be a smooth projective variety of dimension $n$, and $D, D'$ be pseudoeffective divisors on $X$. Recall that two divisors $D, D'$ are \emph{numerically equivalent} and write $D \equiv D'$ if and only if $D\cdot C=D'\cdot C$ for every irreducible curve $C$ on $X$.
It can be easily checked that even if we consider only the curves through a fixed point, it still defines the same numerical equivalence relation.
A correct definition for local numerical equivalence inspired from the theory of Okounkov bodies is suggested by Ro\'{e} (\cite[Definition 4]{Roe}) on surfaces using the Zariski decompositions.
In higher dimensions, we instead use the divisorial Zariski decomposition which can be considered as a natural generalization of the Zariski decomposition in dimension $2$.
Let $D=P+N$ be the divisorial Zariski decomposition.
For a fixed point $x \in X$, we further decompose the negative part $N=N_x + N_x^c$ into the effective divisors $N_x$, $N^c_x$ such that every irreducible component of $N_x$ passes through $x$ and $x \not\in \Supp(N_x^c)$. We say that the decomposition
$$
D=P+N_x+N_x^c
$$
is the \emph{refined divisorial Zariski decomposition of $D$ at a point $x$}.

\begin{definition}
Let $D, D'$ be pseudoeffective divisors on a smooth projective variety $X$ with the refined divisorial Zariski decompositions $D=P+N_x+N_x^c, ~D'=P'+N_x' +{N_x^c}'$ at a point $x \in X$. We say that $D, D'$ are \emph{numerically equivalent near $x$} and write $D \equiv_x D'$ if $P \equiv P'$ and $N_x = N_x'$.
\end{definition}

Proposition \ref{prop_localnumequiv} presents other equivalent conditions of the local numerical equivalence. Clearly, $D\equiv_x D'$ for a point $x$ does not necessarily imply $D\equiv D'$. See Section \ref{locnum_sec} for more details.

\medskip

To extract the local numerical properties of divisors from the Okounkov bodies, it is necessary to consider the Okounkov bodies defined on higher birational models as well.
Thus we consider the following admissible flags on higher birational models.

\begin{definition}[{cf. \cite[Definition 2]{Roe}}]\label{def of flags}
Let $f \colon \widetilde{X} \to X$ be a birational morphism between smooth projective varieties of dimension $n$, and $x \in X$ be a point.
An admissible flag $\widetilde{Y}_\bullet$ on $\widetilde{X}$ is said to be \emph{centered at $x$} if $f(\widetilde{Y}_n)=\{ x\}$.
An admissible flag $\widetilde{Y}_\bullet$ on $\widetilde{X}$ is said to be \emph{proper} (respectively, \emph{infinitesimal}) over $X$ if $\codim f(\widetilde{Y}_i) = i$ (respectively, $\codim f(\widetilde{Y}_1)=n$).
\end{definition}

The first main result of this paper gives a generalization of \cite[Theorem 1]{Roe} into higher dimensions.

\begin{theoremalpha}\label{main1}
Let $D, D'$ be big divisors on a smooth projective variety $X$, and $x \in X$ be a point. Then the following are equivalent:
\begin{enumerate}[wide, labelindent=3pt]
\item[$(1)$] $D \equiv_{x} D'$, that is, $D, D'$ are numerically equivalent near $x$.
\item[$(2)$] $\okbd_{\widetilde{Y}_{\bullet}}(f^*D)=\okbd_{\widetilde{Y}_{\bullet}}(f^*D')$ for every admissible flag $\widetilde{Y}_{\bullet}$ centered at $x$ defined on a smooth projective variety $\widetilde{X}$ with a birational morphism $f \colon \widetilde{X} \to X$.
\item[$(3)$] $\okbd_{\widetilde{Y}_{\bullet}}(f^*D)=\okbd_{\widetilde{Y}_{\bullet}}(f^*D')$ for every proper admissible flag $\widetilde{Y}_{\bullet}$ over $X$ centered at $x$ defined on a smooth projective variety $\widetilde{X}$ with a birational morphism $f \colon \widetilde{X} \to X$.
\item[$(4)$] $\okbd_{\widetilde{Y}_{\bullet}}(f^*D)=\okbd_{\widetilde{Y}_{\bullet}}(f^*D')$ for every infinitesimal flag $\widetilde{Y}_{\bullet}$ over $X$ centered at $x$ defined on a smooth projective variety $\widetilde{X}$ with a birational morphism $f \colon \widetilde{X} \to X$.
\end{enumerate}
\end{theoremalpha}

It is worth noting that proper or infinitesimal admissible flags are sufficient to determine the local numerical properties of a given divisor even though there are admissible flags on higher birational models that are neither proper nor infinitesimal in higher dimensions.

A new ingredient of the proof of Theorem \ref{main1} is the systematic  usage of admissible flags defined on higher birational models of $X$ that are \emph{induced} from a given admissible flag $Y_\bullet$ on $X$ (see Sections \ref{okbd_sec} and \ref{proof_sec}).
Let $f\colon \widetilde{X}\to X$ be a birational morphism with $\widetilde{X}$ smooth projective. Under a suitable condition, there is an obvious proper admissible flag $Y'_\bullet$ on $\widetilde{X}$ over $X$ satisfying $f(Y'_i)=Y_i$. We call such $Y_\bullet'$ an \emph{induced proper admissible flag over $X$}.
We also consider the admissible flag $\widetilde{Y}_\bullet$ on $\widetilde{X}$ such that $\widetilde{Y}_1=E$ is an $f$--exceptional prime divisor and $\widetilde{Y}_{i}$ for $2 \leq i\leq n$ satisfy $\widetilde{Y}_i=E \cap Y_{i-1}'$. We call such $\widetilde{Y}_\bullet$ \emph{an induced infinitesimal admissible flag over $X$}. Under suitable conditions, induced infinitesimal admissible flags are guaranteed to exist on  higher birational models of $X$ (see Lemma \ref{lem-indinf}).
These induced admissible flags on higher birational models of $X$ also play important roles in the proof of Theorem \ref{main2}.
Another ingredient of the proof of Theorem \ref{main1} is Lemma \ref{lem-P.C2}, which is a generalization of Jow's result \cite[Corollary 3.3]{Jow}.

We remark that the implication $(1) \Rightarrow (2)$ of Theorem \ref{main1} was shown in \cite[Proposition 5]{Roe} under the assumption that $D, D'$ admit the Zariski decompositions. Blum--Merz \cite{BM} (and Blum--Malara--Merz--Szpond \cite{BMMS}) also independently obtained the equivalence $(1) \Leftrightarrow (3)$ of Theorem \ref{main1} using a different method.

\medskip

To extract the information of a pseudoeffective divisor $D$ on a smooth projective variety $X$ from the set of Okounkov bodies of $D$ with respect to admissible flags on $X$ centered at a point $x \in X$, we further decompose the divisor $N_x$ in the refined divisorial Zariski decomposition of $D$ at $x$ as
$N_x = N_x^{sm}+N_x^{sing}$
where every irreducible component of $N_x^{sm}$ (respectively, $N_x^{sing}$) is smooth (respectively, singular) at $x$. Then we have a decomposition of a pseudoeffective divisor $D$ as
\begin{equation}\tag{$\star$}\label{sharp}
D = P + N_x^{sm} +N_x^{sing}+ N_x^c.
\end{equation}
Using this further refinement of divisorial Zariski decompositions, we prove the higher dimensional generalization of \cite[Theorem 2]{Roe}.

\begin{theoremalpha}\label{main2}
Let $D, D'$ be big divisors on a smooth projective variety $X$.
For a fixed point $x\in X$, consider the decompositions as in \emph{(\ref{sharp})}
$$
D=P+N_x^{sm} + N_x^{sing} + N_x^c,\;\; D'=P'+{N'}_x^{sm} + {N'}_x^{sing} + {N'}_x^c.
$$
Then the following are equivalent:
\begin{enumerate}[wide, labelindent=3pt]
\item[$(1)$] $P \equiv P', N_x^{sm}={N'}_x^{sm}, \okbd_{Y_\bullet}(N_x^{sing}) = \okbd_{Y_\bullet}({N'}_x^{sing})$ for every admissible flag $Y_\bullet$ centered at $x$.
\item[$(2)$] $\okbd_{Y_\bullet}(D)=\okbd_{Y_\bullet}(D')$ for every admissible flag $Y_\bullet$ on $X$ centered at $x$.
\item[$(3)$] $\okbd_{\widetilde{Y}_\bullet}(f^*D)=\okbd_{\widetilde{Y}_\bullet}(f^*D')$ for every induced proper admissible flag $\widetilde{Y}_\bullet$ over $X$ centered at $x$ defined on a smooth projective variety $\widetilde{X}$ with a birational morphism $f \colon \widetilde{X} \to X$.
\item[$(4)$] $\okbd_{\widetilde{Y}_\bullet}(f^*D)=\okbd_{\widetilde{Y}_\bullet}(f^*D')$ for almost every induced infinitesimal admissible flag $\widetilde{Y}_\bullet$ over $X$ centered at $x$ defined on a smooth projective variety $\widetilde{X}$ with a birational morphism $f \colon \widetilde{X} \to X$.
\end{enumerate}
\end{theoremalpha}

For the surface case, the statement of Theorem \ref{main2} is slightly different from \cite[Theorem 2]{Roe}. In \cite{Roe}, the notion of clusters of infinitely near points plays a crucial role in the proof.
Our notion of induced admissible flags in higher dimensions replaces the role of clusters in the surface case.

Note that $\okbd_{Y_\bullet}(N_x^{sing})$ in the statement (1) of Theorem \ref{main2} consists of a single valuative point and it reflects some singularity properties of $N_x^{sing}$.
Example \ref{ex-main2} shows that the condition $\okbd_{Y_\bullet}(N_x^{sing}) = \okbd_{Y_\bullet}({N'}_x^{sing})$ in the statement (1) of Theorem \ref{main2} is strictly weaker than the condition $N_x^{sing} = {N'}_x^{sing}$ in general.
We also see in Example \ref{ex-main2} that the condition (4) does not hold for arbitrary induced infinitesimal admissible flags. See Theorem \ref{thm:main2intext} for the precise statement of (4).

By Theorem \ref{main1}, Theorem \ref{main2}, and Remark \ref{rem_conseq}, we can see how naturally the local positivity properties (e.g, $x \in \bm(D), x \in \bp(D)$,  $\epsilon(||D||;x)$) are encoded in the Okounkov bodies with respect to admissible flags centered at a point $x$.

\medskip

As a consequence of Theorem \ref{main2}, we can recover the well known result of Jow \cite[Theorem A]{Jow} (see also \cite[Proposition 4.1]{lm-nobody}), which states that if $D, D'$ are big divisors on a smooth projective variety $X$, then
$$
D \equiv D'~\Longleftrightarrow~\okbd_{Y_\bullet}(D)=\okbd_{Y_\bullet}(D')~\text{ for all admissible flags $Y_\bullet$ on $X$.}
$$
Our proof of Theorem \ref{main2} does not use \cite[Theorem A]{Jow}, but we use some ideas in the paper \cite{Jow}.

\medskip

For a pseudoeffective divisor $D$, rather than following the original construction of Okounkov bodies, by taking the limiting procedure on the Okounkov bodies of big divisors near $D$, we can associate to $D$ the so-called \emph{limiting Okounkov body}. We refer to \cite{CHPW-okbd I} for more details. In Section \ref{oklim_sec}, we extend our main results above for big divisors to pseudoeffective divisors (see Theorem \ref{oklim-main1} and Theorem \ref{oklim-main2}). The proofs of Theorem \ref{main1} and Theorem \ref{main2} still work in the pseudoeffective case with little modification.

\medskip

The rest of the paper is organized as follows.
We start in Section \ref{locnum_sec} by defining the local numerical equivalence of pseudoeffective divisors.
In Section \ref{okbd_sec}, we recall the definition of the Okounkov body, and prove some technical results.
Section \ref{proof_sec} is devoted to the proofs of Theorem \ref{main1} and Theorem \ref{main2}.
Finally, the extension of the main results to the limiting Okounkov bodies of pseudoeffective divisors is given in Section \ref{oklim_sec}.

\medskip

Throughout the paper, we work over the field $\C$ of complex numbers. Every divisor is assumed to be an $\R$-Cartier $\R$-divisor.

\subsection*{Acknowledgement}
We are grateful to Joaquim Ro\'{e} for interesting discussions and valuable suggestions.
We would like to thank Harold Blum and Georg Merz for helpful discussions and for sharing their preprint \cite{BM} with us.
We sincerely appreciate referees for valuable suggestions and comments.

\section{Local numerical equivalence}\label{locnum_sec}

In this section, we introduce the local numerical equivalence of pseudoeffective divisors, and prove some basic results. We also recall some basic notions of asymptotic invariants of divisors.

Let $X$ be a smooth projective variety, and $D$ be a pseudoeffective divisor on $X$.
For an irreducible closed subvariety $V\subseteq X$, let $\ord_V(D)$ denote the order of an effective divisor $D$ along $V$.
If $D$ is a big divisor on $X$, then \emph{the asymptotic valuation} of $V$ at $D$ is defined as
$$
\ord_V(||D||):=\inf\{\ord_V(D')\mid D\equiv D'\geq 0\}.
$$
If $D$ is only pseudoeffective, then \emph{the asymptotic valuation} of $V$ at $D$ is defined as
$$
\ord_V(||D||):=\lim_{\epsilon\to 0+}\ord_V(||D+\eps A||)
$$
where $A$ is an ample divisor on $X$. One can check that this definition is independent of the choice of $A$. Note that $\ord_V(||D||)$ is a numerical invariant of $D$ and $\ord_V(|| \cdot ||) \colon \text{Big}(X) \to \R$ is a continuous function on the cone of big divisor classes.
In particular, if $D$ is big, then $\ord_V(||D||)=\lim_{\epsilon\to 0+}\ord_V(||D+\eps A||)$ for any ample divisor $A$ on $X$.

The \emph{divisorial Zariski decomposition} of a pseudoeffective divisor $D$ is the decomposition
$$
D=P+N = P(D) + N(D)
$$
into the \emph{negative part} $N$ defined as
$
N=N(D):=\sum_{E}  \ord_E(||D||)E
$
where the summation is over all the finitely many prime divisors $E$ of $X$ such that $ \ord_E(||D||)>0$ and the \emph{positive part} $P$ defined as
$
P=P(D):=D-N(D).
$
The positive part $P$ can be characterized as the maximal movable divisor such that $P \leq D$ (see \cite[Proposition III.1.14]{nakayama}). By construction, the negative part $N$ is a numerical invariant of $D$. For more details, we refer to \cite{B1}, \cite[Chapter III]{nakayama}.

Following Ro\'{e} \cite{Roe}, for a given point $x\in X$, we further decompose the negative part $N$ as
$$
N=N_x + N_x^c
$$
into the effective divisors $N_x$ and $N_x^c$ such that every irreducible component of $N_x$ passes through $x$ and $x \not\in \Supp(N_x^c)$. We say that
$$
D=P+N_x+N_x^c
$$
is the \emph{refined divisorial Zariski decomposition of $D$ at a point $x$}.

The following is the main result of this section.

\begin{proposition}\label{prop_localnumequiv}
Let $D$ and $D'$ be pseudoeffective divisors on a smooth projective variety $X$ with the refined divisorial Zariski decompositions
$$
D=P+N_x + N_x^c ~~\text{ and }~~ D'=P'+N_x' + N_x'^c
$$
at a point $x \in X$.
Then the following are equivalent:
\begin{enumerate}[wide, labelindent=3pt]
\item[$(1)$] $P \equiv P'$ and $N_x = N_x'$.
\item[$(2)$] $P \equiv P'$ and $\ord_V(||D||)=\ord_V(||D'||)$ for every irreducible subvariety $V \subseteq X$ containing  $x$.
\item[$(3)$] For any birational morphism $f \colon \widetilde{X} \to X$ with $\widetilde{X}$ smooth projective and any point $x'\in f^{-1}(x)$, if we write the refined divisorial Zariski decompositions
$$
f^*D = \widetilde{P} + \widetilde{N}_{x'} + \widetilde{N}_{x'}^c ~~\text{ and }~~ f^*D'=\widetilde{P}' + \widetilde{N}'_{x'} + \widetilde{N}_{x'}'^c
$$
at a point $x'$, then $\widetilde{P} \equiv \widetilde{P}'$ and $\widetilde{N}_{x'} = \widetilde{N}_{x'}'$.
\end{enumerate}
\end{proposition}

\begin{proof}
The implications $(1)\Leftarrow (2) \Leftrightarrow (3)$ are clear. Thus we only have to check the implication $(1)\Rightarrow (2)$.
Suppose that $(1)$ holds, i.e., $P \equiv P'$ and $N_x = N_x'$. It is equivalent to that $D-N_x^c \equiv D' - {N_x'}^c$. Let $V \subseteq X$ be an irreducible subvariety passing through $x$.
Since $x \not\in \Supp(N_x^c) \cup \Supp({N_x'}^c)$, it follows that $\ord_{V}(N_x^c)=\ord_{V}({N_x'}^c)=0$. We then observe that
$$
\begin{array}{l}
\ord_V(||D||)= \ord_V(||D-N_x^c||) + \ord_V(||N_x^c||)=\ord_V(||D-N_x^c||)\\
\ord_V(||D'||)=\ord_V(||D'-{N_x'}^c||) + \ord_V(||{N_x'}^c||)=\ord_V(||D'-{N_x'}^c||),
\end{array}
$$
which implies that $\ord_V(||D||) = \ord_V(||D'||)$.
\end{proof}

\begin{remark}
The condition $\ord_V(||D||)=\ord_V(||D'||)$ for every irreducible subvariety $V \subseteq X$ passing through $x$ in condition (2) is clearly stronger than the condition $N_x=N'_x$ in (1). However, the condition $P\equiv P'$ takes care of this difference.
\end{remark}

The following is a natural generalization of \cite[Definition 4]{Roe}.

\begin{definition}\label{def-locnum}
Under the notations as in Proposition \ref{prop_localnumequiv}, we say two pseudoeffective divisors \emph{$D, D'$ are numerically equivalent near a point $x$} and write $D \equiv_x D'$
if any one of the equivalent conditions of Proposition \ref{prop_localnumequiv} holds.
\end{definition}

It is clear that $D\equiv D'$ implies $D\equiv_x D'$ for any fixed point $x \in X$, but the converse does not hold in general. Notice that $D\equiv D'$ if and only if $D\equiv_x D'$ for all points $x\in X$.

\begin{remark}[Birational invariance]
Note that the numerical equivalence relations are preserved under pull-backs. The same holds for the local numerical equivalence by Proposition \ref{prop_localnumequiv}.
\end{remark}

We now recall the asymptotic base loci of a divisor $D$ on a smooth projective variety $X$.
The \emph{stable base locus} of $D$ is defined as
$$
\SB(D):= \bigcap_{D \sim_{\R} D' \geq 0} \Supp(D').
$$
where $D\sim_\R D'$ denotes the $\R$-linear equivalence, that is, $D-D'$ is the $\R$-linear combination of principal divisors.
The \emph{augmented base locus} of $D$ is defined as
$$
\bp(D):=\bigcap_{A\text{:ample}}\text{SB}(D-A).
$$
The \emph{restricted base locus} of $D$ is defined as
$$
\bm(D):=\bigcup_{A\text{:ample}}\SB(D+A).
$$
It is well known that  $\bp(D)$ and $\bm(D)$ depend only on the numerical class of $D$. Note that $\bm(D)=X$  (respectively, $\bp(D)=X$) if and only if $D$ is not pseudoeffective (respectively, not big), and $\bm(D)=\emptyset$ (respectively, $\bp(D)=\emptyset$) if and only if $D$ is nef (respectively, ample). Since we have $\bm(D)=\bigcup_{\ord_V(||D||)>0} V$ by \cite[Theorem B]{elmnp-asymptotic inv of base}, the union of codimension one components of $\bm(D)$ coincides with  $\Supp(N)$ where $D=P+N$ is the divisorial Zariski decomposition.
For more details on $\bp,\bm$, we refer to \cite{elmnp-asymptotic inv of base}.

We also recall the (restricted) volume of a divisor.  Consider a subvariety $V \subseteq X$ of dimension $v$. The \emph{restricted volume} of $D$ along $V$ is defined as
$$
\vol_{X|V}(D):=\limsup_{m \to \infty} \frac{h^0(X|V,\lfloor mD \rfloor)}{m^v/v!}
$$
where $h^0(X|V, \lfloor mD \rfloor)$ is the dimension of the image of the natural restriction map
$$
\varphi \colon H^0(X, \mc O_X(\lfloor mD \rfloor ))\to H^0(V,\mc O_V(\lfloor mD \rfloor|_V)).
$$
If $V \not\subseteq \bp(D)$, then the restricted volume $\vol_{X|V}(D)$ depends only on the numerical class of $D$, and it uniquely extends to a continuous function
$
\vol_{X|V} \colon \text{Big}^V (X) \to \R
$
where $\text{Big}^V(X)$ is the set of all $\R$-divisor classes $\xi$ such that $V$ is not properly contained in any irreducible component of $\bp(\xi)$. When $V=X$, we simply let $\vol_X(D):=\vol_{X|X}(D)$, and we call it the \emph{volume} of $D$.
For more details, we refer to \cite{elmnp-restricted vol and base loci} and \cite[Section 2]{CHPW-okbd I}

Although the following seems to be well known to experts, we include it here for the completeness in the literature.

\begin{proposition}\label{bp(P)=bp(D)}
Let $X$ be a smooth projective variety, and $D$ be a pseudoeffective divisor on $X$ with the divisorial Zariski decomposition $D=P+N$. Then $\bp(P)=\bp(D)$.
\end{proposition}

\begin{proof}
If $D$ is a non-big pseudoeffective divisor, then so is $P$. Thus $\bp(D)=\bp(P)=X$.
Assuming now that $D$ is big, we first show $\bp(P) \subseteq \bp(D)$. Let $V$ be an irreducible component of $\bp(P)$. By \cite[Theorem C]{elmnp-restricted vol and base loci}, $\vol_{X|V}(P)=0$, and hence, $\vol_{X|V}(D)=0$. By applying \cite[Theorem C]{elmnp-restricted vol and base loci} again, we see that $V \subseteq \bp(D)$. Thus $\bp(P) \subseteq \bp(D)$.

To derive a contradiction, we assume that the inclusion is strict $\bp(P)\subsetneq\bp(D)$. There exists a point $x \in \bp(D) \setminus \bp(P)$. We divide into two cases: (1) $x \not\in \text{Supp}(N)$ and (2) $x \in \text{Supp}(N)$. Suppose that we are in Case (1).
There exists an irreducible component $W$ of $\bp(D)$ containing $x$. By \cite[Theorem C]{elmnp-restricted vol and base loci}, $\vol_{X|W}(D)=0$. Since $W \not\subseteq \text{Supp}(N)$, it follows that $\vol_{X|W}(P)=0$. However, $W \not\subseteq \bp(P)$, so we get a contradiction to \cite[Theorem C]{elmnp-restricted vol and base loci}.
Suppose now that we are in Case (2). Recall that the moving Seshadri constant $\epsilon(||P||;x)$ is positive because $x \not\in \bp(P)$ (\cite[p.644]{elmnp-restricted vol and base loci}). By \cite[Theorem 6.2]{elmnp-restricted vol and base loci},
$\epsilon(|| \cdot ||;x) \colon N^1(X)_{\R} \to \R_{\geq 0}$ is a continuous function. Thus $\epsilon(||P+\epsilon N||;x)>0$ for any sufficiently small $\epsilon >0$. On the other hand, since $P + \epsilon N$ is the divisorial Zariski decomposition by \cite[Lemma III.1.8]{nakayama}, we obtain $x \in \bm(P+\epsilon N) \subseteq \bp(P+\epsilon N)$. Thus $\epsilon(||P+\epsilon N||;x)=0$, which is a contradiction. We complete the proof.
\end{proof}

\begin{remark}\label{rem_conseq}
Recall that if $D \equiv_x D'$, then $V\subseteq\bm(D)$ if and only if $V\subseteq\bm(D')$ for any irreducible subvariety $V$ containing $x$.
Furthermore, the local numerical equivalence class of a pseudoeffective divisor $D$ at a point $x$ determines other various positivity invariants of $D$ such as the moving Seshadri constant $\epsilon(||D||;x)$, the Nakayama constant $\mu(D;x)$, the augmented restricted volume $\vol_{X|V}^+(D)$, and the augmented base locus $\bp(D)$.
\end{remark}

\section{Okounkov bodies and admissible flags}\label{okbd_sec}

In this section, we first review the definition and basic properties of Okounkov bodies. We then study various admissible flags introduced in Definition \ref{def of flags} and related issues, and prove some technical lemmas that are used in the proofs of Theorem \ref{main1} and Theorem \ref{main2}.

Let $X$ be a projective variety of dimension $n$. Recall that an \emph{admissible flag} $Y_\bullet$ on $X$ is a sequence of irreducible subvarieties
$$
Y_\bullet: X=Y_0\supseteq Y_1\supseteq\cdots \supseteq Y_{n-1}\supseteq Y_n=\{x\}
$$
where each $Y_i$ has codimension $i$ in $X$ and is smooth at the point $x$. Let $D$ be a divisor on $X$ with $|D|_{\R}:=\{ D' \mid D \sim_{\R} D' \geq 0 \}\neq\emptyset$.
In the following, we define a valuation-like function
$$
\nu_{Y_\bullet}:|D|_{\R}\to \R_{\geq0}^n.
$$
For any $D'\in |D|_\R$, we let
$
\nu_1=\nu_1(D'):=\ord_{Y_1}(D').
$
Then $D'-\nu_1 Y_1$ is effective and does not contain $Y_1$ in the support, so we can define
$
\nu_2=\nu_2(D'):=\ord_{Y_2}((D'-\nu_1Y_1)|_{Y_1}).
$
Similarly, for $2 \leq i \leq n-1$, we inductively define
$
\nu_{i+1}=\nu_{i+1}(D'):=\ord_{Y_{i+1}}((\cdots((D'-\nu_1Y_1)|_{Y_1}-\nu_2Y_2)|_{Y_2}-\cdots-\nu_iY_i)|_{Y_{i}}).
$
By collecting $\nu_i$'s, we finally obtain
$$
\nu_{Y_\bullet}(D'):=(\nu_1(D'),\nu_2(D'),\cdots,\nu_n(D')) \in \R^n_{\geq 0}.
$$

\begin{definition}\label{def-okbd}
Let $X$ be a projective variety of dimension $n$, and $D$ be a divisor on $X$ such that $|D|_{\R} \neq \emptyset$. The \emph{Okounkov body} $\okbd_{Y_\bullet}(D)$ of $D$ with respect to an admissible flag $Y_\bullet$ on $X$ is a convex subset of $\R^n$ defined as
$$
 \okbd_{Y_\bullet}(D):=\text{the closure of the convex hull of $\nu_{Y_\bullet}(|D|_{\R})$ in $\R^n_{\geq 0}$}.
$$
A point in $\nu_{Y_\bullet}(|D|_{\R})$ is called a \emph{valuative point}.
If $|D|_{\R} = \emptyset$, then we simply let $\okbd_{Y_\bullet}(D) := \emptyset$.
\end{definition}

By \cite[Proposition 4.1]{lm-nobody}, the Okounkov bodies are numerical in nature, i.e., if $D, D'$ are numerically equivalent big divisors, then $\okbd_{Y_\bullet}(D)=\okbd_{Y_\bullet}(D')$ for every admissible flag $Y_\bullet$. 

Note that this definition is equivalent to the one given in~\cite{lm-nobody}, \cite{KK} where the above function $\nu_{Y_\bullet}$ is defined and applied to the nonzero sections $s$ of each $H^0(X,\mc O_X(\lfloor mD \rfloor ))$ of the graded section ring $\bigoplus_{m \geq 0} H^0(X,\mc O_X(\lfloor mD \rfloor))$ and the Okounkov body $\okbd_{Y_\bullet}(D)$ is defined as the convex closure of the set of rescaled images $\frac{1}{m}\nu_{Y_\bullet}(s)$ in $\R^n$.
This equivalent construction can be generalized to a graded linear (sub)series $W_\bullet$ given by a divisor on $X$ to construct the Okounkov body $\okbd_{Y_\bullet}(W_\bullet)$ associated to $W_\bullet$ with respect to an admissible flag $Y_\bullet$.
Now, let $W_\bullet=W_\bullet(D|Y_{n-k})$ be a graded linear series given by a divisor $D$ on $X$ restricted to $Y_{n-k}$ where
$$W_m:=W_m(D|Y_{n-k})= \im\big[H^0\big(X, \mathcal{O}_X(\lfloor mD \rfloor)\big) \to H^0\big(Y_{n-k}, \mathcal{O}_{Y_{n-k}}(\lfloor mD \rfloor|_{Y_{n-k}})\big)\big]
$$
for each $m>0$.
We may regard the partial admissible flag
$$
Y_{n-k \bullet} : Y_{n-k} \supseteq Y_{n-k-1} \supseteq \cdots \supseteq Y_{n-1}\supseteq Y_n=\{x\}
$$
as an admissible flag on $Y_{n-k}$ that is a $k$-dimensional projective variety.
We define the Okounkov body of $D$ with respect to $Y_{n-k \bullet}$ as
$$
\okbd_{Y_{n-k \bullet}}(D):=\okbd_{Y_{n-k \bullet}}(W_\bullet) \subseteq \R^{k}_{\geq 0}\cong \{ 0 \}^{n-k} \times \R^{k}_{\geq 0} \subseteq \R^n_{\geq 0}.
$$
We often regard it as a subset of $\R^n_{\geq 0}$;
By \cite[(2.7) in p.804]{lm-nobody}, if $D$ is a big divisor, then we have $$\vol_{\R^{k}}(\okbd_{Y_{n-k \bullet}}(D)) = \vol_{X|Y_{n-k}}(D).$$
For more details, we refer to \cite{lm-nobody},\cite{KK},\cite{CHPW-okbd I}.

The following theorem is useful to prove Theorem \ref{main1} and Theorem \ref{main2}.

\begin{theorem}[{\cite[Theorem 1.1]{CPW-okbd II}}]\label{slice}
Let $X$ be a smooth projective variety of dimension $n$, and $D$ be a big divisor on $X$. Fix an admissible flag $Y_\bullet$ on $X$ such that $Y_{n-k} \not\subseteq \bp(D)$. Then we have
$$
\okbd_{Y_{n-k \bullet}}(D) =  \okbd_{Y_\bullet}(D)_{x_1=\cdots =x_{n-k}=0} := \okbd_{Y_\bullet}(D) \cap ( \{ 0 \}^{n-k} \times \R^k_{\geq 0} ).
$$
\end{theorem}

The following result tells us that the shape of the Okounkov body is determined by the positive part of the divisorial Zariski decomposition.

\begin{lemma}[cf. {\cite[Lemma 3.9]{CHPW-asyba}, \cite[Theorem C]{AV-loc pos2}}]\label{lem-divzd}
Let $X$ be a smooth projective variety, and $D=P+N$ be the divisorial Zariski decomposition of a big divisor $D$ on $X$. Fix an admissible flag $Y_\bullet$ on $X$. Then we have
$$
\okbd_{Y_\bullet}(D)=\okbd_{Y_\bullet}(P) + \okbd_{Y_\bullet}(N).
$$
Furthermore, $\okbd_{Y_\bullet}(D)=\okbd_{Y_\bullet}(D-E) + \okbd_{Y_\bullet}(E)$ for every effective divisor $E$ with $E \leq N$.
\end{lemma}

\begin{proof}
The first assertion is nothing but \cite[Lemma 3.9]{CHPW-asyba} and  \cite[Theorem C]{AV-loc pos2}.
Since $\okbd_{Y_\bullet}(N)$ consists of a single valuative point in $\R_{\geq 0}^n$, it follows that $\okbd_{Y_\bullet}(N)=\okbd_{Y_\bullet}(N-E) + \okbd_{Y_\bullet}(E)$.
Now, observe that $D-E=P+(N-E)$ is the divisorial Zariski decomposition. Thus we have
$$
\okbd_{Y_\bullet}(D-E) + \okbd_{Y_\bullet}(E)= \okbd_{Y_\bullet}(P) + \okbd_{Y_\bullet}(N-E) + \okbd_{Y_\bullet}(E) = \okbd_{Y_\bullet}(P) + \okbd_{Y_\bullet}(N)=\okbd_{Y_\bullet}(D).
$$
This finishes the proof.
\end{proof}

Next, we define various types of admissible flags.

\begin{definition}\label{def-flags}
Let $X$ be a smooth projective variety of dimension $n$, and $x \in X$ be a point.
Consider a birational morphism $f \colon \widetilde{X} \to X$ from another smooth projective variety $\widetilde{X}$.
\begin{enumerate}[wide, labelindent=3pt]
\item An admissible flag $\widetilde{Y}_\bullet$ on $\widetilde{X}$ is said to be \emph{centered at $x$} if $f(\widetilde{Y}_n)=\{ x \}$.
\item An admissible flag $\widetilde{Y}_\bullet$ on $\widetilde{X}$ is said to be \emph{proper over $X$} if $\codim f(\widetilde{Y}_i)=i$ holds for each $0\leq i \leq n$.
\item An admissible flag $\widetilde{Y}_\bullet$ on $\widetilde{X}$ which is proper over $X$ is said to be \emph{induced} (\emph{by an admissible flag $Y_\bullet$ on $X$}) if  $f(\widetilde{Y}_i) = Y_i$ for each $0\leq i \leq n$.
\item An admissible flag $\widetilde{Y}_\bullet$ on $\widetilde{X}$ is said to be \emph{infinitesimal over $X$} if $f(\widetilde{Y}_1)$ is a point.
\item An admissible flag $\widetilde{Y}_\bullet$ on $\widetilde{X}$ which is infinitesimal over $X$ is said to be \emph{induced} (\emph{by an admissible flag $Y_\bullet$ on $X$}) if there is a proper admissible flag $Y_\bullet'$ on $\widetilde{X}$ induced by $Y_\bullet$ such that $f(\widetilde{Y}_1)=Y_n$ and $\widetilde{Y}_i = \widetilde{Y}_1 \cap Y_{i-1}'$ for $2 \leq i \leq n$. Note that $\widetilde{Y}_n = Y_n'$.
\end{enumerate}
\end{definition}

To show the existence of induced proper/infinitesimal admissible flags, we introduce the following.

\begin{definition}\label{def-admismor}
Let  $f \colon \widetilde{X} \to X$ be a birational morphism between smooth projective varieties of dimension $n$, and $Y_\bullet$ be an admissible flag on $X$. We consider the strict transforms $\widetilde{Y}_{i+1}:=(f|_{\widetilde{Y}_i})_*^{-1} Y_{i+1}$ where $\widetilde{Y}_0:=\widetilde{X}$ and $f|_{\widetilde{Y}_i} \colon \widetilde{Y}_i \to Y_i$ is a birational morphism for $0 \leq i \leq n-1$. Let $\Gamma$ be an effective divisor on $X$, and $\nu_{Y_\bullet}(\Gamma)=(\nu_1, \cdots, \nu_n)$.
\begin{enumerate}[wide, labelindent=3pt]
\item  We say $f$ is a \emph{$Y_\bullet$--admissible morphism} if $\widetilde{Y}_\bullet$ is an admissible flag on $\widetilde{X}$, i.e., each $\widetilde{Y}_i$ is smooth  at the point $\widetilde{Y}_n$.
\item We say $f$ is a \emph{$Y_\bullet$--admissible log resolution of $(X, \Gamma)$} if each $f|_{\widetilde{Y}_i} \colon \widetilde{Y}_i \to Y_i$ is a log resolution of $(Y_i, \Gamma_i + Y_{i+1})$ for $0 \leq i \leq n-1$, where $\Gamma_0:=\Gamma$ and $\Gamma_i:=(\Gamma_{i-1}-\nu_i Y_i)|_{Y_i}$ for $1 \leq i \leq n-1$.
\end{enumerate}
\end{definition}

\begin{example}\label{ex-y-admissible}
We use the notations in Definition \ref{def-admismor}.\\[1pt]
(1) If $f \colon \widetilde{X} \to X$ is isomorphic over a neighborhood of $Y_n$, then it is $Y_\bullet$--admissible. In this case,
$\okbd_{Y_\bullet}(D)=\okbd_{\widetilde{Y}_\bullet}(f^*D)$ (cf. \cite[Lemmas 3.4 and 3.5]{CHPW-asyba}).\\[1pt]
(2) If $f \colon \widetilde{X} \to X$ is a composite of blow-ups of points, then $f$ is $Y_\bullet$--admissible for any admissible flag $Y_\bullet$ on $X$. If furthermore each $Y_i$ is smooth for $1 \leq i \leq n$, then $f$ is a $Y_\bullet$--admissible log resolution of $(X, 0)$. \\[1pt]
(3) Let $f \colon \widetilde{X} \to X$ be the blow-up of a smooth projective $3$-fold $X$ along a smooth projective curve $C$. Suppose that there is an admissible flag $Y_\bullet$ on $X$ such that locally around $Y_3$, the following holds: $Y_1 = \mathbb{A}^2_{x,y}$ is an affine space whose origin is $Y_3$, $C \cap Y_1 = V(x^2, y)$, and $Y_2$ is a general line passing through $Y_3$.
Then $\widetilde{Y}_1$ is singular at $\widetilde{Y}_3$, so $f$ is not $Y_\bullet$--admissible.\\[1pt]
(4) Let $D_1, \cdots, D_{n}$ be effective divisors on $X$ such that $Y_0=X, Y_i:=D_1 \cap \cdots \cap D_i$ with $1 \leq i \leq n$ form an admissible flag $Y_\bullet$ on $X$.
Let $\Gamma$ be an effective divisor.
Then any log resolution $f \colon \widetilde{X} \to X$ of $(X, D_1+ \cdots + D_{n}+\Gamma)$ is a $Y_\bullet$--admissible log resolution of $(X, \Gamma)$.\\[1pt]
(5) Let $Y_\bullet$ be an admissible flag on $X$, and take a birational morphism $\varphi \colon X' \to X$ which is the composition of successive embedded resolutions of singularities of $Y_{n-1}, \cdots, Y_1$. If we denote $Y'_0:=X'$ and $Y_{i+1}':=(\varphi|_{Y_i'})^{-1}_*Y_{i+1}$ where $\varphi|_{Y_i'} \colon Y_i' \to Y_i$ is a birational morphism for each $0 \leq i \leq n-1$, then $Y_\bullet'$ is a proper admissible flag on $X'$ induced by $Y_\bullet$, and all $Y_i'$ are smooth.
Now let $E_1:=Y_1'$ on $X'$, and $h_2 \colon X_2 \to X'$ be the blow-up of $X'$ along $Y_2'$ with exceptional divisor $E_2$. By abuse of notation, we denote the strict transformation $h_{2*}^{-1}E_1$ of $E_1$ on $X_2$ also by $E_1$.
For $3 \leq i \leq n$, let $h_i \colon X_i \to X_{i-1}$ be the blow-up of $X_{i-1}$ along $\big((h_2 \circ \cdots \circ h_{i-1})|_{E_1 \cap \cdots \cap E_{i-1}}\big)_*^{-1} Y_i'$ with exceptional divisor $E_i$, where by abuse of notation again $E_j$ denotes the strict transform on $X_{i-1}$ of $E_j$ for each $1 \leq j \leq i-1$.
Let $\widetilde{X}:=X_n$, and $E_j$ be the strict transforms on $\widetilde{X}$ of $E_j$ for $1 \leq j \leq n$.
Then $g := \varphi \circ h_2 \circ \cdots \circ h_n \colon \widetilde{X} \to X$ is a birational morphism, and $Y_\bullet''$ is a proper admissible flag on $\widetilde{X}$ induced by $Y_\bullet$, where $Y_0''=\widetilde{X}, Y_i'':=E_1 \cap \cdots \cap E_i$ for $1 \leq i \leq n$.
Let $\Gamma$ be an effective divisor.
If $f' \colon \widetilde{X}' \to \widetilde{X}$ is a log resolution of $(\widetilde{X}, E+E_1+g^*\Gamma)$ where $E$ is the sum of all irreducible exceptional divisors over $X$, then $f:=g \circ f' \colon \widetilde{X}' \to X$ is a $Y_\bullet$--admissible log resolution of $(X, \Gamma)$. We may assume that $f$ factors through the blow-up of $X$ at $Y_n$.
\end{example}

If $f$ is a $Y_\bullet$--admissible morphism, then the admissible flag $\widetilde{Y}_\bullet$ on $\widetilde{X}$ consisting of subvarieties $\widetilde{Y}_i$ for $0 \leq i \leq n$ is a proper admissible flag induced by $Y_\bullet$. Conversely, if $\widetilde{Y}_\bullet$ is a proper admissible flag on $\widetilde{X}$ over $X$ such that $Y_i=f(\widetilde{Y}_i)$ for $0 \leq i \leq n$ form an admissible flag $Y_\bullet$ on $X$, then $f$ is a $Y_\bullet$--admissible morphism and $\widetilde{Y}_\bullet$ is induced by $Y_\bullet$. Thus we may say that $Y_\bullet$ and $\widetilde{Y}_\bullet$ determine each other.
By the following lemma, the same is true for the corresponding Okounkov bodies:
if $\widetilde{Y}_\bullet$ is a proper admissible flag on $\widetilde{X}$ induced by $Y_\bullet$ and $D$ is any divisor on $X$, then
$$
\text{
$\okbd_{Y_\bullet}(D)$ and $\okbd_{\widetilde{Y}_\bullet}(f^*D)$ determine each other.}
$$
Note that we do not need to assume that the subvarieties $Y_i$ and $\widetilde{Y}_i$ are smooth for $0 \leq i \leq n$.

\begin{lemma}[{\cite[Lemma 4.1]{CPW-okbd abundant}}]\label{lem-bir}
Let $f \colon \widetilde{X} \to X$ be a birational morphism between smooth projective varieties of dimension $n$, and $\widetilde{Y}_\bullet$ be a proper admissible flag on $\widetilde{X}$ induced by an admissible flag $Y_\bullet$ on $X$.
For $1 \leq i \leq n$, we write $f|_{\widetilde{Y}_{i-1}}^*Y_i = \widetilde{Y}_i + E_i$ for some effective $f|_{\widetilde{Y}_{i-1}}$--exceptional divisor $E_i$ on $\widetilde{Y}_{i-1}$.
If $\mathbf{x}=(x_1, \cdots, x_n)=\nu_{Y_\bullet}(D)$ is a valuative point of an effective divisor $D$ on $X$, then we have
$$
\nu_{\widetilde{Y}_\bullet}(f^*D)=\mathbf{x}+ \sum_{i=1}^{n-1} x_i \cdot \nu_{\widetilde{Y}_{i\bullet}}(E_i|_{\widetilde{Y}_i})
$$
where we regard $\nu_{\widetilde{Y}_{i\bullet}}(E_i|_{\widetilde{Y}_i}) \in \R^{n-i}$ as a point in $\{ 0 \}^i \times \R^{n-i} \subseteq \R^n$.
In particular, for any divisor $D$ on $X$, we have
$$
\okbd_{\widetilde{Y}_\bullet}(f^*D)=\left\{\left. \mathbf{x}+ \sum_{i=1}^{n-1} x_i \cdot \nu_{\widetilde{Y}_{i\bullet}}(E_i|_{\widetilde{Y}_i})\;~\right| \mathbf{x}=(x_1, \cdots, x_n) \in \okbd_{Y_\bullet}(D)\right\}.
$$
\end{lemma}

Now, we show the existence of induced infinitesimal admissible flags under suitable assumptions.

\begin{lemma}\label{lem-indinf}
Let $Y_\bullet$ be an admissible flag on a smooth projective variety $X$ centered at a point $x \in X$, and $f \colon \widetilde{X} \to X$ be a $Y_\bullet$--admissible log resolution of $(X, 0)$ between smooth projective varieties. If $f$ factors through the blow-up of $X$ at $x$, then there exists a unique infinitesimal admissible flag $\widetilde{Y}_\bullet$ on $\widetilde{X}$ induced by $Y_\bullet$.
\end{lemma}

\begin{proof}
Let $n:=\dim X$, and denote by $Y_\bullet'$ the proper admissible flag on $\widetilde{X}$ induced by $Y_\bullet$.
We consider the birational morphism
$
f_i:=f|_{Y_{i}'} \colon Y_{i}' \to Y_{i}
$
between projective varieties of dimension $n-i$ for each $0 \leq i \leq n-2$.
We claim that for each $0 \leq i \leq n-2$, there exists a unique $f|_{Y_{i}'}$--exceptional prime divisor $E_i$ on $Y_{i}'$ such that $f_i(E_i)=\{ x \}$, the variety $E_i \cap Y_{n-1}'$ consists of a single point $x'$, and $E_i \cap Y_{j-1}'$ is an irreducible subvariety of codimension $j-i$ in $Y_{i}'$ and is smooth at $x'$ for $i+2 \leq j \leq n-1$. The claim for $i = 0$ implies that if we let $\widetilde{Y}_0:=\widetilde{X}, \widetilde{Y}_1:=E,$ and $\widetilde{Y}_i:=E \cap Y_{i-1}'$ for $2 \leq i \leq n$, then $\widetilde{Y}_\bullet$ is a unique infinitesimal admissible flag on $\widetilde{X}$ induced by $Y_\bullet$.
To prove the claim, we proceed by induction on the dimension of $Y_i$. The claim is trivial for the surface case where $i=n-2$. We suppose that for a positive integer $k \leq n-2$, the claim holds for all $i$ with $k \leq i \leq n-2$.
Then we can find an $f_{k-1}$--exceptional prime divisor $E_{k-1}$ on $Y_{k-1}'$ such that $E_{k-1}|_{Y_{k}'}=E_k$.
Since $f_{k-1}$ is a log resolution of $(Y_{k-1}, Y_{k})$, it follows that $\text{exc}(f_{k-1}) \cup Y_k'$ has a simple normal crossing support on $Y_{k-1}'$. Thus the divisor $E_{k-1}$ on $Y_{k-1}'$ is uniquely determined. Moreover, it is straightforward to check that this divisor $E_{k-1}$ on $Y_{k-1}'$ satisfies all required properties for applying the induction. We have shown the claim, so we complete the proof.
\end{proof}

\begin{example}\label{ex-indinf}
Let $f \colon \widetilde{X} \to X$ be the blow-up of a smooth projective variety $X$ of dimension $n$ at a point $x \in X$ with the exceptional divisor $E$, and $E_\bullet$ be an infinitesimal admissible flag over $x$ in the sense of \cite[Definition 2.1]{AV-loc pos} (cf. \cite[Section 5]{lm-nobody}), i.e., $E_0 = \widetilde{X}, E_1=E$ and $E_i$ is an $(n-i)$-dimensional linear subspace of $E \cong \P^{n-1}$ for $2 \leq i \leq n$. We claim that $E_\bullet$ is an induced infinitesimal admissible flag over $X$. There is a smooth hypersurface $Y_1 \subseteq X$ such that $\widetilde{Y}_1 \cap E = E_2$ where $\widetilde{Y}_1:=f_*^{-1}Y_1$. Inductively, for $3 \leq i \leq n$, we can find a smooth hypersurface $Y_{i-1} \subseteq Y_{i-2}$ such that $\widetilde{Y}_{i-1} \cap E = E_{i}$ where $\widetilde{Y}_{i-1}:=(f|_{\widetilde{Y}_{i-2}})^{-1}_* Y_{i-1}$. By letting $Y_0=X, Y_n=\{ x \}$, the subvarieties $Y_i$ form an admissible flag $Y_\bullet$ on $X$. Then $f$ is a $Y_\bullet$--admissible log resolution of $(X, 0)$, and $E_\bullet$ is an infinitesimal admissible flag induced by the admissible flag $Y_\bullet$.
On the other hand, there are many other admissible flags on $X$ that induce the same infinitesimal admissible flag $E_\bullet$.
\end{example}

\begin{example}\label{ex-nonind}
For a smooth projective variety $X$ of dimension $n$, there exists a proper or infinitesimal admissible flag on a higher birational model $\widetilde{X} $ of $X$ that is not induced by any admissible flag on $X$.
If $\widetilde{Y}_\bullet$ is a proper admissible flag on $\widetilde{X}$ and $f(\widetilde{Y}_i)$ is singular at $f(\widetilde{Y}_n)$, then it is clearly not induced over $X$. For the infinitesimal case, let $f \colon \widetilde{X} \to X$ be the composite of the blow-ups at a point $x$ and an infinitely near point to $x$ with the exceptional divisor $E$ and $E'$, respectively. Then any infinitesimal admissible flag $\widetilde{Y}_\bullet$ on $\widetilde{X}$ satisfying $\widetilde{Y}_1=E'$ and $\widetilde{Y}_2=E' \cap E$ is not induced over $X$.
\end{example}

The following is useful in the study of the Okounkov bodies with respect to induced infinitesimal admissible flags. It is a counterpart of the first assertion in Lemma \ref{lem-bir} under a stronger assumption.

\begin{lemma}\label{lem-bir-indinf}
Let $Y_\bullet$ be an admissible flag on a smooth projective variety $X$ of dimension $n$, and $D$ be an effective divisor on $X$.
Let $f \colon \widetilde{X} \to X$ be a $Y_\bullet$--admissible log resolution of $(X, D)$, and $\widetilde{Y}_\bullet$ (respectively, $Y_\bullet'$) be an infinitesimal (respectively, a proper) admissible flag on $\widetilde{X}$ induced by $Y_\bullet$.
If $\nu_{Y_\bullet'}(f^*D+E)=(x_1, \cdots, x_{n-1}, x_n)$
for an $f$--exceptional effective divisor $E$, then we have $\nu_{\widetilde{Y}_\bullet}(f^*D+E)=(x_n, x_1, \cdots, x_{n-1})$.
In particular, $\nu_{Y_\bullet'}(f^*D+E)$ and $\nu_{\widetilde{Y}_\bullet}(f^*D+E)$ determine each other.
\end{lemma}

\begin{proof}
Let $D_0:=f^*D+E$, and $D_i:=(D_{i-1}- x_i Y_i')|_{Y_i'}$ for $1 \leq i \leq n-1$.
Note that each $D_i+Y_{i+1}'$ has a simple normal crossing support on $Y_i'$ for $0 \leq i \leq n-1$. If $x_{i+1}>0$ for $0 \leq i \leq n-1$, then there is a unique irreducible component $E_{i+1}$ in $f^*D+E$ such that $Y_{i+1}' \subseteq E_{i+1}$ and $Y_i' \not\subseteq E_{i+1}$. In this case, we have $\mult_{E_{i+1}}(f^*D+E)=x_{i+1}$. If $x_{i+1}=0$, then put $E_{i+1}:=0$. We write
$$
f^*D+E=x_1E_1 + \cdots + x_nE_n + \big(f^*D+E-(x_1E_1 + \cdots + x_nE_n) \big).
$$
Then we have $\widetilde{Y}_n=Y_n' \not\subseteq \Supp\big(f^*D+E-(x_1E_1 + \cdots + x_nE_n) \big)$. Notice that if $x_n>0$, then $E_n = \widetilde{Y}_1$.
Now, it follows that $\nu_{\widetilde{Y}_\bullet}(f^*D+E)=(x_n, x_1, \cdots, x_{n-1})$.
\end{proof}

\begin{remark}
It is impossible to have an analogous statement of Lemma \ref{lem-bir} for induced infinitesimal admissible flags.
Let $S$ be a smooth projective surface with a very ample divisor $D$, and $f \colon \widetilde{S} \to S$ be the blow-up of $S$ at a point $x \in S$ with the exceptional divisor $E$.
Suppose that there is an irreducible curve $C$ on $S$ such that $\epsilon(D; x)=\frac{D \cdot C}{\mult_{x} C} < \sqrt{D^2}$ and $(f_*^{-1}C)^2<0$. We can choose smooth curves $Y_1, Y_1' \in |D|$ passing through $x$ such that $f_*^{-1}Y_1 \cap f_*^{-1}C \cap E= \emptyset$ and $f_*^{-1}Y_1' \cap f_*^{-1}C \cap E \neq \emptyset$. Consider admissible flags $Y_\bullet: S \supseteq Y_1 \supseteq \{x \}$ and  $Y_\bullet': S \supseteq Y_1' \supseteq \{x \}$ on $S$.
Note that $\okbd_{Y_\bullet}(D)=\okbd_{Y_\bullet'}(D)$ and $\okbd_{\widetilde{Y}_\bullet}(f^*D) = \okbd_{\widetilde{Y}_\bullet'}(f^*D)$ for proper admissible flags $\widetilde{Y}_\bullet, \widetilde{Y}_\bullet'$ on $\widetilde{S}$ induced by $Y_\bullet, Y_\bullet'$, respectively. However, we can check that $\okbd_{\widetilde{Y}_\bullet}(f^*D) \neq \okbd_{\widetilde{Y}_\bullet'}(f^*D)$ for infinitesimal admissible flags $\widetilde{Y}_\bullet, \widetilde{Y}_\bullet'$ on $\widetilde{S}$ induced by $Y_\bullet, Y_\bullet'$.

It is also impossible to determine the Okounkov body with respect to an induced infinitesimal admissible flag on a higher birational model $\widetilde{X}$ of $X$ by using the set of the Okounkov bodies with respect to admissible flags on $X$.
Let $S_1$ be a very general K3 surface of degree $6$ in $\P^4$ with a hyperplane section $D_1$, and $(S_2, D_2)$ be a very general polarized abelian surface of type $(1,3)$. For each $i=1,2$, fix a very general point $x_i \in S_i$, and take the blow-up $f_i \colon \widetilde{S}_i \to S_i$ of $S_i$ at $x_i$ with the exceptional divisor $E_i$. We define the following sets
\[
\begin{array}{l}
\okbd_i:=\{ \okbd_{Y_{\bullet}}(D_i) \mid Y_{\bullet} \text{ is an admissible flag on $S_i$ centered at $x_i$} \} \\
\okbd_i':=\{ \okbd_{Y_\bullet'}(f_i^*D_i) \mid Y_\bullet' \text{ is an induced proper flag on $\widetilde{S}_i$ centered at $x_i$} \} \\
\widetilde{\okbd}_i:=\{ \okbd_{E_\bullet}(f_i^* D_i) \mid E_\bullet \text{ is an induced infinitesimal admissible flag on $\widetilde{S}_i$ centered at $x_i$}\}.
\end{array}
\]
It is easy to see that $\okbd_1=\okbd_2$ and $\okbd_1'=\okbd_2'$ as sets.
However, $\epsilon(D_1; x_1)=2$ by \cite[Theorem 1.2]{GK} and $\epsilon(D_2; x_2)=\frac{12}{5}$ by \cite[Theorem 6.1]{B}, and it follows from  \cite[Theorem C]{AV-loc pos} that the size of the maximal inverted simplex contained in $\okbd_{E_\bullet}(f_i^*D_i)$ is $\epsilon(D_i;x_i)$ for any infinitesimal admissible flag $E_\bullet$ on $\widetilde{S}_i$ centered at $x_i$. Thus we see that $\widetilde{\okbd}_1 \cap \widetilde{\okbd}_2 = \emptyset$.
\end{remark}

\section{Proofs of main results}\label{proof_sec}

In this section, we prove Theorem \ref{main1} and Theorem \ref{main2}. We start by showing some lemmas that are the key ingredients of the proofs. For the lemmas, we use the following notations: $X$ is a smooth projective variety of dimension $n$, and $D$ is a big divisor on $X$ with the refined divisorial Zariski decomposition $D=P+N_x + N^c_x$ at a point $x \in X$.

First, we explain how to recover $N_x$ from the Okounkov bodies of $D$ with respect to admissible flags centered at $x$. We remark that Lemma \ref{lem-Nprop} is already observed in \cite[Proof of Theorem A]{Jow}.

\begin{lemma}\label{lem-Nprop}
For an irreducible component $E$ of $N_x$ such that $E$ is smooth at $x$, we have
$$
\mult_E N_x =  \inf_{Y_\bullet} \left\{ x_1 \left| (x_1, \cdots, x_n) \in \okbd_{Y_\bullet}(D)\right. \right\}
$$
where $\inf$ is taken over all admissible flags $Y_\bullet$ on $X$ centered at $x$ with $Y_1=E$.
\end{lemma}

\begin{proof}
The right hand side is $\ord_E(||D||)$, and by definition, $\ord_E(||D||)=\mult_E N_x$.
\end{proof}

\begin{lemma}\label{lem-Ninf}
Let $E$ be an irreducible component of $N_x$ such that $E$ is smooth at $x$, and $\Gamma$ be an effective divisor on $X$ with $E \not\subseteq \Supp(\Gamma)$. Then we have
$$
\mult_E N_x = \inf_{\widetilde{Y}_\bullet} \left\{ x_2 \left| (x_1,x_2, \cdots, x_n) \in \okbd_{\widetilde{Y}_\bullet}(f^*D)\right. \right\}
$$
where $\inf$ is taken over all infinitesimal admissible flags on $\widetilde{X}$ induced by admissible flags $Y_\bullet$ on $X$  centered at $x$ with $Y_1=E$ where $f \colon \widetilde{X} \to X$ is a $Y_\bullet$--admissible log resolution of $(X, \Gamma)$.
\end{lemma}

\begin{proof}
For simplicity, we denote by $\alpha$ the value on the right hand side in the lemma. We first show that $\mult_E N_x \leq \alpha$.
Let $\widetilde{Y}_\bullet$ be an infinitesimal admissible flag on $\widetilde{X}$ induced by an admissible flag $Y_\bullet$ on $X$ with $Y_1=E$ where $f \colon \widetilde{X} \to X$ is a birational morphism between smooth projective varieties.
Note that every effective divisor in $|f^*D|_{\mathbb R}$ has the form $f^*D_0 = f^{-1}_*D_0 + F$ for some $D_0 \in |D|_{\mathbb R}$ and for some $f$--exceptional effective divisor $F$.
Let $\nu_{\widetilde{Y}_\bullet}(f^*D_0)=(\nu_1, \cdots, \nu_n)$ be a valuative point of $\okbd_{\widetilde{Y}_\bullet}(f^*D)$. We have
$$
f^*D_0 - \nu_1 \widetilde{Y}_1 = f^{-1}_*D_0 + (F-\nu_1 \widetilde{Y}_1) = f^{-1}_*(D_0-N_x-N_x^c) + f^{-1}_*(N_x+N_x^c) + (F-\nu_1 \widetilde{Y}_1).
$$
The divisor $F-\nu_1 \widetilde{Y}_1$ is effective since $\widetilde{Y}_1$ is $f$--exceptional and is not a component of $f^{-1}_*D_0$. Clearly the divisors $f^{-1}_*(D_0-N_x-N_x^c), f^{-1}_*(N_x+N_x^c)$ are also effective. Thus $f^*D_0 - \nu_1 \widetilde{Y}_1\geq f^{-1}_* N_x$. Since $\widetilde{Y}_2 = \widetilde{Y}_1 \cap f_*^{-1} E$, it follows that
$$
\nu_2=\ord_{\widetilde{Y}_2}((f^*D_0 - \nu_1 \widetilde{Y}_1)|_{\widetilde{Y}_1}) \geq \ord_{\widetilde{Y}_2}(f^{-1}_*(N_x)|_{\widetilde{Y}_1}) \geq \mult_E N_x.
$$
This implies that $\mult_E N_x \leq \alpha$.

To show the equality $\mult_E N_x= \alpha$, let $\epsilon>0$ be any positive number.
By \cite[III. 1.4 Lemma (5)]{nakayama}, we can find some $P_0\in|P|_{\mathbb R}$ such that $0\leq \mult_{E}P_0< \epsilon$.
Now, fix an admissible flag $Y_\bullet$ on $X$ centered at $x$ with $Y_1=E$, and take a $Y_\bullet$--admissible log resolution $f \colon \widetilde{X} \to X$ of $(X,  P_0 + N_x + N_x^c)$ which factors through the blow up of $X$ at $x$.
 By Lemma \ref{lem-indinf}, there is an infinitesimal admissible flag $\widetilde{Y}_\bullet$ on $\widetilde{X}$ induced by $Y_\bullet$.
Let $\nu_{\widetilde{Y}_\bullet}(f^*(P_0+N_x + N_x^c)) = (\nu_1, \cdots, \nu_n)$ be a valuative point of $\okbd_{\widetilde{Y}_\bullet}(f^*D)$.
Since the effective divisor $D_0':=f^*(P_0+N_x + N_x^c) - \nu_1 \widetilde{Y}_1$ has a simple normal crossing support and $\widetilde{Y}_2 = \widetilde{Y}_1 \cap f_*^{-1}E$, it follows that
\begin{align*}
\nu_2&=\ord_{\widetilde{Y}_2}({D_0'}|_{\widetilde{Y}_1})=\mult_{f_*^{-1}E }(D_0')\\
&=\mult_{E} (P_0 + N_x + N_x^c)=\mult_{E}P_0 + \mult_E N_x < \epsilon+ \mult_E N_x.
\end{align*}
This implies that $\mult_E N_x \leq \alpha \leq \mult_E N_x + \epsilon$.
Since $\epsilon>0$ can be chosen arbitrarily small, the equality $\mult_E N_x= \alpha$ actually holds.
\end{proof}

Next, we explain how to recover the positive part $P$ from the Okounkov bodies of $D$ with respect to admissible flags centered at $x$. For this purpose, we recall the following basic lemma.

\begin{lemma}[{\cite[Lemma 3.5]{Jow}}]\label{curveclass}
Let $X$ be a smooth projective variety of dimension $n$ with $\rho=\dim N^1(X)_{\Q}$. If $Y \subseteq X$ is a transversal complete intersection of $n-2$ general very ample effective divisors and $H_1, \cdots, H_{\rho}$ are very ample effective divisors on $X$ whose numerical classes form a basis of $N^1(X)_{\Q}$, then the set of curve classes
$\{[C_i := Y \cap H_i]|\;i=1, \cdots, \rho\}$ forms a basis of $N_1(X)_{\Q}$.
\end{lemma}

Note that we may allow all the curves $C_i$ in Lemma \ref{curveclass} to be smooth projective curves and pass through a given point $x \in X$. Suppose that we can read off the intersection numbers $P \cdot C_i$ from the Okounkov bodies of a divisor $P$. Then, by Lemma \ref{curveclass}, we can determine the numerical equivalence class of $P$. It is a crucial step in the proofs of Theorem \ref{main1} and Theorem \ref{main2} to recover the intersection numbers $P \cdot C_i$ from the Okounkov bodies of $P$.

The following can be regarded as a stronger version of Jow's result \cite[Corollary 3.3]{Jow}.

\begin{lemma}\label{lem-P.C2}
Let $Y_\bullet$ be an admissible flag on $X$ such that $Y_{n-1} \not\subseteq \bp(P)$ and $Y_{n-1} \cap \bm(P)=\emptyset$. Then we have
$$
P\cdot Y_{n-1}=\vol_{\R^1}(\okbd_{Y_\bullet}(P) \cap \text{$x_n$\text{--}$\axis$})=\vol_{\R^1}(\okbd_{Y_\bullet}(P)_{x_1=\cdots=x_{n-1}=0}).
$$
\end{lemma}

\begin{proof}
Since we have $\vol_{\R^1}(\okbd_{Y_\bullet}(P)_{x_1=\cdots=x_{n-1}=0}) = \vol_{X|Y_{n-1}}(P)$ by Theorem \ref{slice} and \cite[(2.7) in p.804]{lm-nobody}, it is sufficient to check that
$$
\vol_{X|Y_{n-1}}(P)=P\cdot Y_{n-1}.
$$
Let $A$ be an ample divisor on $X$. For any sufficiently small real number $\epsilon > 0$, we have $\SB(P + \epsilon A) \cap Y_{n-1} = \emptyset$ and $Y_{n-1} \not\subseteq \bp(P+\epsilon A)$. By \cite[Theorem B]{elmnp-restricted vol and base loci}, we see that $ \vol_{X|Y_{n-1}}(P+\epsilon A) = (P+\epsilon A) \cdot Y_{n-1}$. Thus we obtain
$$
\lim_{\epsilon \to 0+} \vol_{X|Y_{n-1}}(P+\epsilon A) = \lim_{\epsilon \to 0+} (P+\epsilon A) \cdot Y_{n-1} = P \cdot Y_{n-1}.
$$
On the other hand, since $\vol_{X|Y_{n-1}} \colon \text{Big}^{Y_{n-1}} (X) \to \R$ is a continuous function by \cite[Theorem 5.2]{elmnp-restricted vol and base loci}, it follows that
$$
\lim_{\epsilon \to 0+} \vol_{X|Y_{n-1}}(P+\epsilon A) = \vol_{X|Y_{n-1}}(P).
$$
Therefore, $\vol_{X|Y_{n-1}}(P)=P\cdot Y_{n-1}$ as desired.
\end{proof}

\begin{lemma}\label{lem-P.Cinf}
Let $\Gamma$ be an effective divisor on $X$, and $Y_\bullet$ be an admissible flag on $X$ centered at a point $x\in X$ such that $Y_{n-1} \not\subseteq \bp(P), Y_{n-1} \cap \bm(P) = \emptyset,$ and $Y_{n-1} \not\subseteq \Supp(\Gamma)$. Then we have
$$
P \cdot Y_{n-1} = \sup_{\widetilde{Y}_\bullet} \vol_{\R^1}(\okbd_{\widetilde{Y}_\bullet}(f^*P) \cap \text{$x_1$--$\axis$}) = \sup_{\widetilde{Y}_\bullet} \left\{ a_1 \left| (a_1,0, \cdots, 0) \in \okbd_{\widetilde{Y}_\bullet}(f^*P)\right. \right\}
$$
where both $\sup$s are taken over all infinitesimal admissible flags on $\widetilde{X}$ induced by the admissible flag $Y_\bullet$ where $f \colon \widetilde{X} \to X$ is a $Y_\bullet$--admissible log resolution of $(X, \Gamma)$.
\end{lemma}

\begin{proof}
We first claim that the second equality holds:
$$
\sup_{\widetilde{Y}_\bullet} \vol_{\R^1}(\okbd_{\widetilde{Y}_\bullet}(f^*P) \cap \text{$x_1$--$\axis$}) = \sup_{\widetilde{Y}_\bullet} \left\{ a_1 \left| (a_1,0, \cdots, 0) \in \okbd_{\widetilde{Y}_\bullet}(f^*P)\right. \right\}.
$$
We will denote this value by $\beta$.
Note that $\widetilde{Y}_n \not\subseteq \bm(f^*P)$.  By \cite[Theorem A]{CHPW-asyba}, the origin of $\R_{\geq 0}^n$ is contained in $\okbd_{\widetilde{Y}_\bullet}(f^*P)$ for any infinitesimal admissible flag $\widetilde{Y}_\bullet$ over $X$ centered at $x$. Thus $(a_1, 0, \cdots, 0) \in \okbd_{\widetilde{Y}_\bullet}(f^*P)$ if and only if $a_1 \leq  \vol_{\R^1}(\okbd_{\widetilde{Y}_\bullet}(f^*P) \cap \text{$x_1$--$\axis$})$, so the claimed equality holds. It only remains to prove that
$$
P \cdot Y_{n-1} = \beta.
$$

First, we show that $P \cdot Y_{n-1} \leq \beta$. By Theorem \ref{slice} and Lemma \ref{lem-P.C2}, we have
$$
\vol_{\R^1}(\okbd_{Y_\bullet}(P) \cap \text{$x_n$--$\axis$}) = \vol_{\R^1}( \okbd_{Y_{n-1 \bullet}}(P)  )=P \cdot Y_{n-1}.
$$
For any positive number $\epsilon > 0$, there is some $P_0 \in |P|_{\R}$ such that
$$
P \cdot Y_{n-1} - \epsilon = \vol_{\R^1}( \okbd_{Y_{n-1 \bullet}}(P)  ) - \epsilon <\nu_{Y_{n-1 \bullet}}(P_0)=:b .
$$
Note that $\nu_{Y_\bullet}(P_0)=(0, \cdots, 0, b) \in \okbd_{Y_\bullet}(P)$. Now, take a $Y_\bullet$--admissible log resolution $f \colon \widetilde{X} \to X$ of $(X, \Gamma + P_0 )$ which factors through the blow up of $X$ at $x$.
Let $Y_\bullet'$ be a proper admissible flag on $\widetilde{X}$ induced by $Y_\bullet$. By Lemma \ref{lem-indinf}, there is  an infinitesimal admissible flag $\widetilde{Y}_\bullet$ on $\widetilde{X}$ induced by $Y_\bullet$. Note that each $f|_{Y_i'} \colon Y_i' \to Y_i$ is also a log resolution of $(Y_i, \Gamma|_{Y_i} + P_0|_{Y_i}+ Y_{i+1})$ for $0 \leq i \leq n-2$.
We see that the only irreducible component of $f^*P_0$ containing $\widetilde{Y}_n$ is $\widetilde{Y}_1$. Now, $f|_{Y_{n-1}'} \colon Y_{n-1}' \to Y_{n-1}$ is an isomorphism over a neighborhood of $Y_n$, so we obtain $b=\ord_x P_0|_{Y_{n-1}} = \ord_{\widetilde{Y}_n} f^*P_0|_{Y_{n-1}'}$. Since $f^*P_0$ has a simple normal crossing support and $f^*P_0$ meets $Y_{n-1}'$ transversally, it follows that $\mult_{\widetilde{Y}_1} f^*P_0 = b$. This implies that $\nu_{\widetilde{Y}_\bullet}(f^*P_0) = (b, 0, \cdots, 0) \in \okbd_{\widetilde{Y}_\bullet}(f^*P)$. Thus we have
$$
P \cdot Y_{n-1} - \epsilon < b \leq \beta.
$$
Since $\epsilon > 0$ can be taken arbitrarily small, this implies that $P \cdot Y_{n-1} \leq \beta$.

To derive a contradiction, suppose that $P \cdot Y_{n-1} < \beta$. There is a $Y_\bullet$--admissible log resolution $f \colon \widetilde{X} \to X$ of $(X, \Gamma)$ and an infinitesimal admissible flag $\widetilde{Y}_\bullet$ on $\widetilde{X}$ induced by the admissible flag $Y_\bullet$ on $X$  such that
$$
P \cdot Y_{n-1} < \vol_{\R^1}(\okbd_{\widetilde{Y}_\bullet}(f^*P) \cap \text{$x_1$--$\axis$}).
$$
Fix an ample divisor $A$ on $X$ and a sufficiently small number $\epsilon > 0$ such that
$$
\vol_{\R^1}(\okbd_{\widetilde{Y}_\bullet}(f^*P) \cap \text{$x_1$--$\axis$}) - (P + \epsilon A) \cdot Y_{n-1} \gg 2\epsilon .
$$
We can take a number $k$ with
$$
(P+\epsilon A) \cdot Y_{n-1}  < k  <\vol_{\R^1}(\okbd_{\widetilde{Y}_\bullet}(f^*P) \cap \text{$x_1$--$\axis$})~~\text{ and }~~k - (P+\epsilon A) \cdot Y_{n-1} > 2\epsilon.
$$
Since $Y_{n-1} \not\subseteq \bp(P+\epsilon A)$ and $Y_{n-1} \cap \bm(P+\epsilon A) = \emptyset$, it follows from Lemma \ref{lem-P.C2} that
\begin{equation}\tag{$\#$}\label{eq-pf}
 \vol_{\R^1}(\okbd_{Y_\bullet}(P+\epsilon A) \cap x_n\text{--}\axis) = (P+\epsilon A) \cdot Y_{n-1} < k  - 2\epsilon.
\end{equation}
Now, choose an effective $f$--exceptional divisor $E$ such that $f^*\epsilon A - E$ is ample.
Notice that the divisor $f^*\epsilon A - \delta E = (1-\delta)f^*\epsilon A + \delta(f^* \epsilon A - E)$ is ample for any number $\delta$ with $0<\delta \leq 1$.
Since the origin of $\R^n_{\geq 0}$ is contained in
$$
\okbd_{\widetilde{Y}_\bullet}(f^*P-k\widetilde{Y}_1) = \okbd_{\widetilde{Y}_\bullet}(f^*P)_{x_1 \geq k} + (-k, \underbrace{0, \cdots, 0}_{n-1}),
$$
it follows from \cite[Theorem A]{CHPW-asyba} that $\widetilde{Y}_n \not\subseteq \bm(f^*P-k\widetilde{Y}_1)$. Then for each $0 < \delta \leq 1$, we have
$$
\widetilde{Y}_n \not\subseteq \text{SB}(f^*P-k\widetilde{Y}_1 + f^*\epsilon A - \delta E),
$$
so there is some $P_{\delta} \in |P+\epsilon A|_{\R}$ such that
$$
f^*P-k\widetilde{Y}_1 + f^*\epsilon A - \delta E \sim_{\R}f^*P_\delta - k\widetilde{Y}_1 - \delta E \geq 0~~\text{ and }~~\widetilde{Y}_n \not\subseteq \Supp(f^*P_\delta - k\widetilde{Y}_1 - \delta E).
$$
We write $f^*P_\delta = (f^*P_\delta - k\widetilde{Y}_1 - \delta E) + k\widetilde{Y}_1 + \delta E$. Let $\nu_{\widetilde{Y}_\bullet}(E) = (\nu_1, \cdots, \nu_n)$.
Then we have
$$
\nu_{\widetilde{Y}_\bullet}(f^*P_\delta) = \nu_{\widetilde{Y}_\bullet}(k\widetilde{Y}_1+\delta E) = (k+\delta \nu_1, \delta \nu_2, \cdots , \delta \nu_n).
$$
Let $Y_\bullet'$ be the proper admissible flag on $\widetilde{X}$ induced by $Y_\bullet$. By Lemma \ref{lem-bir-indinf}, we have
$$
\nu_{Y_\bullet'}(f^*P_\delta) = \nu_{Y_\bullet'}(k\widetilde{Y}_1 + \delta E) = (\delta \nu_2, \cdots, \delta \nu_n, k+\delta \nu_1).
$$
Let $\nu_{Y_\bullet}(P_\delta)=(\nu_1'(\delta), \cdots, \nu_n'(\delta)) \in \okbd_{Y_\bullet}(P+\epsilon A)$. In view of Lemma \ref{lem-bir}, we can take the nonnegative numbers $\nu_1'(\delta), \cdots, \nu_{n-1}'(\delta), (k+\delta \nu_1)-\nu_n'(\delta)$ arbitrarily small by taking $\delta$ sufficiently small. Thus if $\delta$ is sufficiently small, then we may assume that
$$
(k+\delta \nu_1) - \nu_n'(\delta) < \epsilon~\text{ and }~ \nu_n'(\delta) < \vol_{\R^1}(\okbd_{Y_\bullet}(P+\epsilon A) \cap x_n\text{--}\axis)  + \epsilon.
$$
Then by (\ref{eq-pf}), we have
$$
k- \epsilon < \nu_n'(\delta) < (P+\epsilon A) \cdot Y_{n-1}+\epsilon < k- \epsilon,
$$
which is a contradiction. Hence $P \cdot Y_{n-1} = \beta$, and we finish the proof.
\end{proof}

Using Lemmas \ref{lem-P.C2} and \ref{lem-P.Cinf}, we can determine the numerical equivalence class of $P$.

\begin{lemma}\label{lem-Pprop}
The numerical equivalence class of $P$ is determined by the set
$$
\{ \left(\okbd_{Y_\bullet}(D), Y_\bullet \right) \mid \text{$Y_\bullet$ is an admissible flag on $X$ centered at $x$} \}.
$$
\end{lemma}

\begin{proof}
It is equivalent to proving that if $\okbd_{Y_\bullet}(D)=\okbd_{Y_\bullet}(D')$ for all admissible flags $Y_\bullet$ on $X$ centered at $x$, then $P \equiv P'$, where $D=P+N, D'=P'+N'$ are the divisorial Zariski decompositions. Let $\pi \colon \text{Bl}_x X \to X$ be the blow-up of $X$ at $x$ with the exceptional divisor $E$, and $\pi^*D=\widetilde{P}+\widetilde{N}, \pi^*D'=\widetilde{P}'+\widetilde{N}'$ be the divisorial Zariski decompositions. Since $E \not\subseteq\bm(\widetilde{P})\cup \bm(\widetilde{P}')$, we can choose a point $x' \in E \setminus \big( \bm(\widetilde{P}) \cup \bm(\widetilde{P}') \big)$.
We can take an admissible flag $Y_\bullet'$ on $X$ centered at $x'$ such that each $Y_i'$ is a smooth projective variety given by a transversal complete intersection of $i$ very general very ample effective divisors on $\text{Bl}_x X$ for $1 \leq i \leq n-1$. We may assume that $Y_{n-1}' \not\subseteq \bp(\widetilde{P}) \cup \bp(\widetilde{P}'), Y_{n-1}' \cap \big(\bm(\widetilde{P}) \cup \bm(\widetilde{P}') \big) = \emptyset$, and $Y_\bullet'$ is a proper admissible flag induced by an admissible flag $Y_\bullet$ on $X$ centered at $x$.
By Lemma \ref{lem-divzd}, $\okbd_{Y_\bullet'}(\widetilde{P})$ and $\okbd_{Y_\bullet'}(\widetilde{P}')$ are translations of  $\okbd_{Y_\bullet'}(\pi^*D)$ and $\okbd_{Y_\bullet'}(\pi^*D')$ in $\R_{\geq 0}^n$, respectively.
But Lemma \ref{lem-bir} says that $\okbd_{Y_\bullet'}(\pi^*D) = \okbd_{Y_\bullet'}(\pi^*D')$, and \cite[Theorem A]{CHPW-asyba} says that the origin of $\R^n_{\geq 0}$ is contained in both $\okbd_{Y_\bullet'}(\widetilde{P})$ and $\okbd_{Y_\bullet'}(\widetilde{P}')$. Thus $\okbd_{Y_\bullet'}(\widetilde{P}) = \okbd_{Y_\bullet'}(\widetilde{P}')$, so Lemma \ref{lem-P.C2} implies that
$\widetilde{P} \cdot Y_{n-1}' = \widetilde{P}' \cdot Y_{n-1}'$.
By applying Lemma \ref{curveclass} with varying $Y_\bullet'$, we see that $\widetilde{P} \equiv \widetilde{P}'$. Consequently, we obtain $P=\pi_*\widetilde{P} \equiv \pi_*\widetilde{P}' = P'$.
This completes the proof.
\end{proof}

\begin{lemma}\label{lem-Pinf}
Let $\Gamma$ be an effective divisor on $X$.
The numerical equivalence class of $P$ is determined by the set
\[
\left\{ \left(\okbd_{\widetilde{Y}_\bullet}(f^*D), \widetilde{Y}_\bullet \right) \left|
\begin{array}{l} \text{$\widetilde{Y}_\bullet$ is an infinitesimal admissible flag on $\widetilde{X}$ induced by an admissible flag}  \\
\text{$Y_\bullet$ on $X$ centered at $x$ such that $f \colon \widetilde{X} \to X$ is a $Y_\bullet$--admissible log} \\
\text{resolution of $(X, \Gamma)$}
\end{array} \right. \right\}.
 \]
\end{lemma}

\begin{proof}
It is equivalent to proving that if $\okbd_{\widetilde{Y}_\bullet}(f^*D)=\okbd_{\widetilde{Y}_\bullet}(f^*D')$ for all infinitesimal admissible flags $\widetilde{Y}_\bullet$ described in the set in the lemma, then $P \equiv P'$, where $D=P+N, D'=P'+N'$ are the divisorial Zariski decompositions. Let $\pi \colon \text{Bl}_x X \to X$ be the blow-up of $X$ at $x$ with the exceptional divisor $E$, and $\pi^*D=\widetilde{P}+\widetilde{N}, \pi^*D'=\widetilde{P}'+\widetilde{N}'$ be the divisorial Zariski decompositions. As in the proof of Lemma \ref{lem-Pprop}, we consider a proper admissible flag $Y_\bullet'$ on $\text{Bl}_x X$ centered at a point $x' \in E \setminus \big( \bm(\widetilde{P}) \cup \bm(\widetilde{P}') \big)$ induced by an admissible flag $Y_\bullet$ on $X$ such that each $Y_i'$ is a smooth projective variety given by a transversal complete intersection of $i$ very general very ample effective divisors on $\text{Bl}_x X$ for $1 \leq i \leq n-1$ and $Y_{n-1}' \not\subseteq \bp(\widetilde{P}) \cup \bp(\widetilde{P}'), Y_{n-1}' \cap \big(\bm(\widetilde{P}) \cup \bm(\widetilde{P}') \big) = \emptyset$. We can further assume that $Y_{n-1}' \not\subseteq \Supp(\pi^*\Gamma)$.
Let $\widetilde{Y}_\bullet$ be an infinitesimal admissible flag on $\widetilde{X}$ induced by $Y_\bullet'$, where $f' \colon \widetilde{X} \to \text{Bl}_x X$ is a $Y_\bullet'$--admissible log resolution of $(\text{Bl}_x X, \pi^*\Gamma+E)$.
Note that $f:=\pi \circ f'$ is a $Y_\bullet$--admissible log resolution of $(X, \Gamma)$, and $\widetilde{Y}_\bullet$ is an infinitesimal admissible flag induced by $Y_\bullet$.
Since $\okbd_{\widetilde{Y}_\bullet}(f^*D)=\okbd_{\widetilde{Y}_\bullet}(f^*D')$
and $\widetilde{Y}_n \not\subseteq \bm(f'^*\widetilde{P}) \cup \bm(f'^*\widetilde{P}')$, it follows that
$\okbd_{\widetilde{Y}_\bullet}(f'^*\widetilde{P}) = \okbd_{\widetilde{Y}_\bullet}(f'^*\widetilde{P}')$ by Lemma  \ref{lem-divzd} and \cite[Theorem A]{CHPW-asyba}.
By considering all possible $\widetilde{Y}_\bullet$, we can derive $\widetilde{P} \cdot Y_{n-1}' = \widetilde{P}' \cdot Y_{n-1}'$ from Lemma \ref{lem-P.Cinf}. By applying Lemma \ref{curveclass} with varying $Y_\bullet'$ , we see that $\widetilde{P} \equiv \widetilde{P}'$. Hence $P=\pi_*\widetilde{P} \equiv \pi_*\widetilde{P}' = P'$, so we complete the proof.
\end{proof}

We are ready to give the proof of Theorem \ref{main1}.

\begin{proof}[Proof of Theorem \ref{main1}]
\noindent $(1)\Rightarrow(2)$:
Let us assume that $D \equiv_x D'$.
Let $\widetilde{Y}_{\bullet}$ be an admissible flag on a smooth projective variety $\widetilde{X}$ centered at $x$ where $f \colon \widetilde{X}\to X$ is a birational morphism, and $\widetilde{Y}_n=\{x'\}$ for some $x'\in f^{-1}(x)$. Denote by $f^*D=P+N_{x'}+N^c_{x'}$ and $f^*D'=P'+N'_{x'}+N'^c_{x'}$ the refined divisorial Zariski decompositions at $x'$. Then, by Proposition \ref{prop_localnumequiv}, we have $P\equiv P'$ and $N_{x'}=N'_{x'}$, so we obtain
$$
\okbd_{\widetilde{Y}_\bullet}(P)  = \okbd_{\widetilde{Y}_\bullet}(P')~~ \text{ and } ~~\okbd_{\widetilde{Y}_\bullet}(N_{x'}) = \okbd_{\widetilde{Y}_\bullet}(N'_{x'}).
$$
Since each of $\okbd_{\widetilde{Y}_\bullet}(N_{x'}+N^c_{x'}), \okbd_{\widetilde{Y}_\bullet}(N'_{x'}+N'^c_{x'})$ consists of a single valuative point in $\R^n_{\geq 0}$ and $x' \not\in \Supp(N^c_{x'})\cup\Supp(N'^c_{x'})$, it follows that
$$
 \okbd_{\widetilde{Y}_\bullet}(N_{x'}+N^c_{x'}) =  \okbd_{\widetilde{Y}_\bullet}(N_{x'}) = \okbd_{\widetilde{Y}_\bullet}(N'_{x'}) = \okbd_{\widetilde{Y}_\bullet}(N'_{x'}+N'^c_{x'}).
$$
By Lemma \ref{lem-divzd}, we have
$$
\okbd_{\widetilde{Y}_\bullet}(f^*D)=\okbd_{\widetilde{Y}_\bullet}(P) + \okbd_{\widetilde{Y}_\bullet}(N_{x'}+N^c_{x'}) = \okbd_{\widetilde{Y}_\bullet}(P') + \okbd_{\widetilde{Y}_\bullet}(N'_{x'}+N'^c_{x'}) =\okbd_{\widetilde{Y}_\bullet}(f^*D').
$$
Thus $(2)$ holds.

\smallskip

\noindent $(2)\Rightarrow(3)$ and $(2)\Rightarrow(4)$: It is obvious.

\smallskip

\noindent $(3)\Rightarrow(1)$ and $(4) \Rightarrow (1)$: Let $D=P+N_x + N_x^c, D'=P'+N_x' + {N_x'}^c$ be the refined divisorial Zariski decompositions of big divisors $D, D'$ at a point $x$.
Recall the conditions $(3)$ and $(4)$:
\begin{enumerate}[wide, labelindent=3pt]
\item [$(3)$] $\okbd_{\widetilde{Y}_{\bullet}}(f^*D)=\okbd_{\widetilde{Y}_{\bullet}}(f^*D')$ for every proper admissible flag $\widetilde{Y}_{\bullet}$ over $X$ centered at $x$ defined on a smooth projective variety $\widetilde{X}$ with a birational morphism $f \colon \widetilde{X} \to X$.
\item [$(4)$] $\okbd_{\widetilde{Y}_{\bullet}}(f^*D)=\okbd_{\widetilde{Y}_{\bullet}}(f^*D')$ for every infinitesimal admissible flag $\widetilde{Y}_{\bullet}$ over $X$ centered at $x$ defined on a smooth projective variety $\widetilde{X}$ with a birational morphism $f \colon \widetilde{X} \to X$.
\end{enumerate}
Under the condition (3) or (4), we want to show that $P \equiv P'$ and $N_x = N_x'$.

We first show that $N_x = N_x'$. Let $E$ be an irreducible component of $N_x$. It is sufficient to show the following claim
$$
\mult_E N_x = \mult_E N_x'.
$$
If $E$ is smooth at $x$, then the claim follows from Lemma \ref{lem-Nprop} under the condition (3) or Lemma \ref{lem-Ninf} under the condition (4).
We now assume that $E$ is singular at $x$. Take a log resolution $f \colon \widetilde{X} \to X$ of $(X, E)$ so that the strict transform $f_*^{-1}E$ is smooth. There exists a point $x'$ in $f_*^{-1}E$ with $f(x')=\{ x \}$. Let $f^*D = \widetilde{P} + \widetilde{N}_{x'} + \widetilde{N}_{x'}^c$ be the refined divisorial Zariski decomposition of $D$ at $x'$. Note that $\mult_E N_x = \mult_{f_*^{-1}E} \widetilde{N}_{x'}$. By Lemma \ref{lem-Nprop}, $\mult_{f_*^{-1}E} \widetilde{N}_{x'}$ is determined by the Okounkov bodies of $f^*D$ with respect to admissible flags $\widetilde{Y}_\bullet$ on $\widetilde{X}$, which is proper over $X$, centered at $x'$ with $\widetilde{Y}_1=f_*^{-1}E$. Thus, under the condition (3), this implies the claim.
For the infinitesimal case, we note that every infinitesimal admissible flag on $\widetilde{X}$ centered at $x'$ is an infinitesimal admissible flag over $X$ centered at $x$. Under the condition (4), by applying Lemma \ref{lem-Ninf}, we also see that the claim holds. Therefore, $N_x = N_x'$.

Now, Lemma \ref{lem-Pprop} under the condition (3) or Lemma \ref{lem-Pinf} under the condition (4) immediately implies $P\equiv P'$. This completes the proof of Theorem \ref{main1}.
\end{proof}

We now turn to the proof of Theorem \ref{main2} in Introduction. Let $D$ be a pseudoeffective divisor on a smooth projective variety $X$, and $D=P+N_x + N_x^c$ be the refined divisorial Zariski decomposition at a point $x \in X$. We can further decompose $N_x$ as
$N_x = N_x^{sm} + N_x^{sing}$
where every irreducible component of $N_x^{sm}$ (respectively, $N_x^{sing}$) is smooth (respectively, singular) at $x$. Then we have a decomposition of a pseudoeffective divisor $D$ as
\begin{equation}\tag{$\star$}\label{sharp2}
D = P + N_x^{sm} +N_x^{sing}+ N_x^c.
\end{equation}

\begin{theorem}[$=$Theorem \ref{main2}]\label{thm:main2intext}
Let $D, D'$ be big divisors on a smooth projective variety $X$.
For a fixed point $x\in X$, consider the decompositions as in \emph{(\ref{sharp2})}
$$
D=P+N_x^{sm} + N_x^{sing} + N_x^c,\;\; D'=P'+{N'}_x^{sm} + {N'}_x^{sing} + {N'}_x^c.
$$
Then the following are equivalent:
\begin{enumerate}[wide, labelindent=3pt]
\item[$(1)$] $P \equiv P', N_x^{sm}={N'}_x^{sm}, \okbd_{Y_\bullet}(N_x^{sing}) = \okbd_{Y_\bullet}({N'}_x^{sing})$ for every admissible flag $Y_\bullet$ centered at $x$.
\item[$(2)$] $\okbd_{Y_\bullet}(D)=\okbd_{Y_\bullet}(D')$ for every admissible flag $Y_\bullet$ on $X$ centered at $x$.
\item[$(3)$] $\okbd_{\widetilde{Y}_\bullet}(f^*D)=\okbd_{\widetilde{Y}_\bullet}(f^*D')$ for every induced proper admissible flag $\widetilde{Y}_\bullet$ over $X$ centered at $x$ defined on a smooth projective variety $\widetilde{X}$ with a birational morphism $f \colon \widetilde{X} \to X$.
\item[$(4)$] $\okbd_{\widetilde{Y}_\bullet}(f^*D)=\okbd_{\widetilde{Y}_\bullet}(f^*D')$ for every infinitesimal admissible flag $\widetilde{Y}_\bullet$ on $\widetilde{X}$ induced by an admissible flag $Y_\bullet$ on $X$ centered at $x$ where $f \colon \widetilde{X} \to X$ is a $Y_\bullet$--admissible log resolution of $(X, N_x^{sing} + {N'}_x^{sing})$.
\end{enumerate}
\end{theorem}

\begin{proof}
\noindent $(1) \Rightarrow (2)$: Note that (1) implies that
$$
\okbd_{Y_\bullet}(P)=\okbd_{Y_\bullet}(P'), \okbd_{Y_\bullet}(N_x^{sm}) = \okbd_{Y_\bullet}({N'}_x^{sm}), \okbd_{Y_\bullet}(N_x^{sing}) = \okbd_{Y_\bullet}({N'}_x^{sing})
$$
for any admissible flag $Y_\bullet$ on $X$. Then (2) follows from Lemma \ref{lem-divzd}.

\smallskip

\noindent $(2) \Rightarrow (1)$: Assume the condition (2) holds. By Lemma \ref{lem-Nprop}, we can recover $N_x^{sm}$ from the Okounkov bodies of $D$ with respect to admissible flags on $X$ centered at $x$, so we have $N_x^{sm}={N'}_x^{sm}$.
Lemma \ref{lem-Pprop} implies that $P \equiv P'$. It then follows from Lemma \ref{lem-divzd} that
$\okbd_{Y_\bullet}(N_x^{sing})= \okbd_{Y_\bullet}({N'}_x^{sing})$
for any admissible flag $Y_\bullet$ on $X$ centered at $x$. Thus (1) holds.

\smallskip

\noindent $(2) \Leftrightarrow (3)$: It follows from Lemma \ref{lem-bir}.

\smallskip

\noindent $(1) \Rightarrow (4)$: By Proposition \ref{prop_localnumequiv} and Lemmas \ref{lem-divzd},  \ref{lem-bir}, and \ref{lem-bir-indinf}, this implication follows (cf. Proof of Theorem \ref{main1} $(1)\Rightarrow(2)$).

\smallskip

\noindent $(4) \Rightarrow (1)$: By Lemmas \ref{lem-Ninf} and \ref{lem-Pinf}, we have $N_x^{sm} = N_x^{sm}$ and $P \equiv P'$. Now, by Lemmas \ref{lem-divzd}, \ref{lem-bir}, and \ref{lem-bir-indinf}, we obtain $\okbd_{Y_\bullet}(N_x^{sing})=\okbd_{Y_\bullet}({N'}_x^{sing})$ for every admissible flag $Y_\bullet$ on $X$ centered at $x$. Thus (1) holds. Therefore, we complete the proof.
\end{proof}

\begin{example}\label{ex-main2}
Let $D, D'$ be big divisors on a smooth projective surface $S$ with the Zariski decompositions $D=P+N, D'=P'+N'$, and $f \colon \widetilde{S} \to S$ be the blow-up of $S$ at a point $x \in S$ with the exceptional divisor $E$. Suppose that $N, N'$ are irreducible curves that are singular at $x$, and the strict transforms $f_*^{-1}N, f_*^{-1}N'$ are smooth but meet $E$ at the two points $p, q$ satisfying
$$
\ord_p (f_*^{-1}N|_E)=2, \ord_q (f_*^{-1}N|_E)=3 ~~\text{and}~~ \ord_p (f_*^{-1}N'|_E)=3, \ord_q (f_*^{-1}N'|_E)=2.
$$
Then it is easy to check that
$$
\text{$\okbd_{Y_\bullet}(N) = \okbd_{Y_\bullet}(N')$ for every admissible flag $Y_\bullet$ on $S$ centered at $x$}
$$
even though $N \neq N'$.
Thus we see that the condition (1) in Theorem \ref{main2} does not necessarily imply that $N_x^{sing}={N'}_x^{sing}$.
Moreover, for an induced infinitesimal admissible flag  $\widetilde{Y}_\bullet: \widetilde{S} \supseteq E \supseteq \{ p \}$ over $S$, we have
$$
\okbd_{\widetilde{Y}_\bullet}(f^*N) \neq \okbd_{\widetilde{Y}_\bullet}(f^*N).
$$
This shows that the conditions (1), (2), (3) in Theorem \ref{main2} do not imply $\okbd_{\widetilde{Y}_\bullet}(f^*D)=\okbd_{\widetilde{Y}_\bullet}(f^*D')$ for every induced infinitesimal admissible flag $\widetilde{Y}_\bullet$ over $X$ centered at $x$ defined on a smooth projective variety $\widetilde{X}$ with a birational morphism $f \colon \widetilde{X} \to X$.
\end{example}

\section{Extension to limiting Okounkov bodies of pseudoeffective divisors}\label{oklim_sec}

Theorem \ref{main1} and Theorem \ref{main2} can be easily extended to pseudoeffective divisors using the limiting Okounkov bodies. First, we recall the definition of the limiting Okounkov body.

\begin{definition}
Let $X$ be a smooth projective variety of dimension $n$, and $D$ be a pseudoeffective divisor on $X$.
The \emph{limiting Okounkov body} $\oklim_{Y_\bullet}(D)$ of $D$ with respect to an admissible flag $Y_\bullet$ is a convex subset of $\R^n$ defined as
$$
\oklim_{Y_\bullet}(D):=\lim_{\epsilon \to 0+} \okbd_{Y_\bullet}(D+\epsilon A) = \bigcap_{\epsilon >0} \okbd_{Y_\bullet}(D+\epsilon A) \text{ in $\R^n_{\geq 0}$}
$$
where $A$ is an ample divisor on $X$.\footnote{The word ``body'' usually means a compact convex set with nonempty interior. Although $\oklim_{Y_\bullet}(D)$ may not satisfy this nonempty interior condition in general, we call it the limiting Okounkov ``body".}
The definition of the limiting Okounkov body $\oklim_{Y_\bullet}(D)$ is independent of the choice of the ample divisor $A$.
\end{definition}

If $D$ is a big divisor, then $\oklim_{Y_\bullet}(D)=\okbd_{Y_\bullet}(D)$.
Note also that the same construction for $\okbd_{Y_\bullet}(D)$ is valid for a pseudoeffective divisor $D$ as long as $|D|_{\mathbb R}\neq \emptyset$. We call such $\okbd_{Y_\bullet}(D)$ a \emph{valuative Okounkov body}. We saw in  \cite{CHPW-okbd I} that the limiting Okounkov bodies $\oklim_{Y_\bullet}(D)$ reflect more naturally the numerical properties of a pseudoeffective divisor $D$ rather than the valuative Okounkov body. We refer to \cite{CHPW-okbd I, CPW-okbd II, CPW-okbd abundant} for more properties.

By slightly modifying the arguments in the proof of Theorem \ref{main1}, we obtain the following.

\begin{theorem}\label{oklim-main1}
Let $D, D'$ be pseudoeffective divisors on a smooth projective variety $X$, and $x \in X$ be a point. Then the following are equivalent:
\begin{enumerate}[wide, labelindent=3pt]
\item[$(1)$] $D \equiv_{x} D'$, that is, $D, D'$ are numerically equivalent near $x$.
\item[$(2)$] $\oklim_{\widetilde{Y}_{\bullet}}(f^*D)=\oklim_{\widetilde{Y}_{\bullet}}(f^*D')$ for every admissible flag $\widetilde{Y}_{\bullet}$ centered at $x$ defined on a smooth projective variety $\widetilde{X}$ with a birational morphism $f \colon \widetilde{X} \to X$.
\item[$(3)$] $\oklim_{\widetilde{Y}_{\bullet}}(f^*D)=\oklim_{\widetilde{Y}_{\bullet}}(f^*D')$ for every proper admissible flag $\widetilde{Y}_{\bullet}$ over $X$ centered at $x$ defined on a smooth projective variety $\widetilde{X}$ with a birational morphism $f \colon \widetilde{X} \to X$.
\item[$(4)$] $\oklim_{\widetilde{Y}_{\bullet}}(f^*D)=\oklim_{\widetilde{Y}_{\bullet}}(f^*D')$ for every infinitesimal flag $\widetilde{Y}_{\bullet}$ over $X$ centered at $x$ defined on a smooth projective variety $\widetilde{X}$ with a birational morphism $f \colon \widetilde{X} \to X$.
\end{enumerate}
\end{theorem}

\begin{proof}
Let $D=P+N_x+N_x^c$ and $D'=P'+N_x'+{N_x'}^c$ be the refined divisorial Zariski decomposition at $x$. Then $\oklim_{\widetilde{Y}_\bullet}(f^*D)=\oklim_{\widetilde{Y}_\bullet}(f^*(D-N_x^c))$ for all admissible flag $\widetilde{Y}_\bullet$ over $X$ centered at $x$ (see \cite[Lemma 3.9]{CHPW-asyba}). By replacing $D, D'$ by $D-N_x, D'-{N_x'}^c$, we may assume that $N_x^c={N_x'}^c=0$.
We fix an ample divisor $A$ on $X$.

\smallskip

\noindent $(1) \Rightarrow (2)$:
For any number $\epsilon >0$, the big divisors $D+\epsilon A, D'+\epsilon A$ are numerically equivalent near $x$. By Theorem \ref{main1}, the Okounkov bodies of pull-backs of $D + \epsilon A$ and $D'+\epsilon A$ coincide for all admissible flags over $X$ centered at $x$. By letting $\epsilon \mapsto 0$, we obtain the implication $(1) \Rightarrow (2)$.

\smallskip

\noindent $(2) \Rightarrow (3)$ and $(2)\Rightarrow (4)$: It is obvious.

\smallskip

\noindent $(3) \Rightarrow (1)$ and $(4) \Rightarrow (1)$:  For any number $\epsilon \geq 0$, let $D+\epsilon A=P^\epsilon + N_x^{\epsilon} + N_x^{c, \epsilon}$ be the refined divisorial Zariski decomposition at $x$. Note that $\lim_{\epsilon \to 0} P^{\epsilon} = P$ and $\lim_{\epsilon \to 0} N_x^{\epsilon} = N_x $.
By Lemma \ref{lem-Nprop} under the condition (3) or Lemma \ref{lem-Ninf} under the condition (4), one can read off $N_x^{\epsilon}$ from the Okounkov bodies of pull-backs of $D + \epsilon A$. By letting $\epsilon \mapsto 0$, we can recover $N_x$ from the limiting Okounkov bodies of pull-backs of $D$. Similarly, using Lemmas \ref{lem-Pprop} and \ref{lem-Pinf}, we can recover $P$ from the limiting Okounkov bodies of pull-backs of $D$. Thus we obtain the implications $(3) \Rightarrow (1)$ and $(4) \Rightarrow (1)$.
\end{proof}

We can similarly prove the following theorem as in the proof of Theorem \ref{main2}. We leave the details of the proofs to the interested readers.

\begin{theorem}\label{oklim-main2}
Let $D, D'$ be pseudoeffective divisors on a smooth projective variety $X$.
For a fixed point $x\in X$, consider the decompositions as in \emph{(\ref{sharp})} in Introduction
$$
D=P+N_x^{sm} + N_x^{sing} + N_x^c,\;\; D'=P'+{N'}_x^{sm} + {N'}_x^{sing} + {N'}_x^c.
$$
Then the following are equivalent:
\begin{enumerate}[wide, labelindent=3pt]
\item[$(1)$] $P \equiv P', N_x^{sm}={N'}_x^{sm}, \okbd_{Y_\bullet}(N_x^{sing}) = \okbd_{Y_\bullet}({N'}_x^{sing})$ for every admissible flag $Y_\bullet$ centered at $x$.
\item[$(2)$] $\oklim_{Y_\bullet}(D)=\oklim_{Y_\bullet}(D')$ for every admissible flag $Y_\bullet$ on $X$ centered at $x$.
\item[$(3)$] $\oklim_{\widetilde{Y}_\bullet}(f^*D)=\oklim_{\widetilde{Y}_\bullet}(f^*D')$ for every induced proper admissible flag $\widetilde{Y}_\bullet$ over $X$ centered at $x$ defined on a smooth projective variety $\widetilde{X}$ with a birational morphism $f \colon \widetilde{X} \to X$.
\item[$(4)$]  $\oklim_{\widetilde{Y}_\bullet}(f^*D)=\oklim_{\widetilde{Y}_\bullet}(f^*D')$ for every infinitesimal admissible flag $\widetilde{Y}_\bullet$ on $\widetilde{X}$ induced by an admissible flag $Y_\bullet$ on $X$ centered at $x$ where $f \colon \widetilde{X} \to X$ is a $Y_\bullet$--admissible log resolution of $(X, N_x^{sing} + {N'}_x^{sing})$.
\end{enumerate}
\end{theorem}

As a consequence of Theorem \ref{oklim-main2}, we can recover one of the main results of \cite{CHPW-okbd I} (cf. \cite[Theorem A]{Jow}), which states that if $D,D'$ are pseudoeffective divisors on a smooth projective variety $X$, then
$$
D \equiv D' ~\Longleftrightarrow~\oklim_{Y_\bullet}(D)=\oklim_{Y_\bullet}(D')~\text{ for all admissible flags $Y_\bullet$ on $X$.}
$$

\end{document}